\documentclass[a4paper,11pt]{article}
\usepackage{amsfonts,amsmath,amssymb,amsthm,verbatim,placeins,xcolor,comment,placeins,caption,subcaption,bm,graphicx}
\usepackage[shortlabels]{enumitem}
\usepackage{geometry}
\usepackage[pdfauthor={DG},pdfstartview={FitH},colorlinks=true,citecolor=blue,hyperindex,breaklinks]{hyperref} 
\usepackage[english]{babel}
\usepackage[sort&compress]{natbib}
\usepackage{pgfplots}

\newtheorem{theorem}{Theorem}[section]
\newtheorem{proposition}{Proposition}[section]
\newtheorem{assumption}{Assumption}
\newtheorem{lemma}{Lemma}[section]

\theoremstyle{remark}

\theoremstyle{definition}

\theoremstyle{remark}
\newtheoremstyle{myremark}{}{}{\color{blue}\small}{}{\color{blue}\bfseries}{}{ }{}

\theoremstyle{myremark}

\newcommand{\E}{\mathbb{E}} 
\renewcommand{\Re}{\operatorname{Re}} 
\newcommand{\R}{{\mathbb R}}

\newcommand{\Z}{{\mathbb Z}}

\newcommand\1{\mathbf{1}}

\DeclareMathOperator{\sign}{sign}
\DeclareMathOperator{\med}{med}

\begin{document}

\renewcommand*{\thefootnote}{\fnsymbol{footnote}}

\begin{center}
\Large{\textbf{Intermittency and infinite variance: the case of integrated supOU processes}}\\
\bigskip
\bigskip
Danijel Grahovac$^1$\footnote{dgrahova@mathos.hr}, Nikolai N.~Leonenko$^2$\footnote{LeonenkoN@cardiff.ac.uk}, Murad S.~Taqqu$^3$\footnote{murad@bu.edu}\\
\end{center}

\bigskip
\begin{flushleft}
\footnotesize{
$^1$ Department of Mathematics, University of Osijek, Trg Ljudevita Gaja 6, 31000 Osijek, Croatia\\
$^2$ School of Mathematics, Cardiff University, Senghennydd Road, Cardiff, Wales, UK, CF24 4AG}\\
$^3$ Department of Mathematics and Statistics, Boston University, Boston, MA 02215, USA
\end{flushleft}

\bigskip

\textbf{Abstract: } SupOU processes are superpositions of Ornstein-Uhlenbeck type processes with a random intensity parameter. They are stationary processes whose marginal distribution and dependence structure can be specified independently. Integrated supOU processes have then stationary increments and satisfy central and non-central limit theorems. Their moments, however, can display an unusual behavior known as ``intermittency''. We show here that intermittency can also appear when the processes have a heavy tailed marginal distribution and, in particular, an infinite variance.

\bigskip


\bigskip


\bigskip

\section{Introduction}\label{sec1}

Superpositions of Ornstein-Uhlenbeck type (supOU) processes provide models with analytically and stochastically tractable dependence structure displaying either weak or strong dependence and also having marginal distributions that are infinitely divisible. They have applications in environmental studies, ecology, meteorology, geophysics, biology, see \cite{barndorff2015recent,barndorff2018ambit,podolskij2015} and the references therein.  The supOU processes are particularly relevant in finance and the statistical theory of turbulence since they can model key stylized features of observational series from finance and turbulence (see e.g.~\cite{barndorff2018ambit,barndorff2004levy,barndorff2001non,barndorff2013multivariate,barndorff2013stochastic,barndorff2005burgers,stelzer2015derivative}). Recently in \cite{kelly2013active}, the supOU processes have even been used to assess the mass of black hole.

SupOU processes form a rich class of stationary processes with a flexible dependence structure. They are defined as integrals with respect to an infinitely divisible random measure (see Section \ref{s:pre}) and their distribution is determined by the \textit{characteristic quadruple}
\begin{equation}\label{quadruple}
(a,b,\mu,\pi),
\end{equation}
where $(a,b,\mu)$ is some L\'evy-Khintchine triplet (see e.g.~\cite{sato1999levy}) and $\pi$ is a probability measure on $\R_+$. In the construction of the supOU process $\{X(t), \ t \in \R\}$, the choice of $(a,b,\mu)$ uniquely characterizes the one-dimensional marginals. These do not depend on the choice of $\pi$. The probability distribution $\pi$ affects the dependence structure however. See Section \ref{s:pre} and \cite{bn2001,barndorff2018ambit,barndorff2011multivariate,barndorff2013multivariate,barndorff2013stochastic,GLST2017Arxiv} for details.

By aggregating the supOU process $\{X(t), \ t \in \R\}$ one obtains the \textit{integrated supOU process}
\begin{equation}\label{integratedsupOU}
X^*(t) = \int_0^t X(s) ds.
\end{equation}
A suitably normalized integrated process exhibits complex limiting behavior. Indeed, if the underlying supOU process has finite variance, then four classes of processes may arise in a classical limiting scheme (\cite{GLT2017Limit}). Namely, the limit process may be Brownian motion, fractional Brownian motion, a stable L\'evy process or a stable process with dependent increments. The type of limit depends on whether the Gaussian component is present in \eqref{quadruple} or not, on the behavior of $\pi$ in \eqref{quadruple} near the origin and on the growth of the L\'evy measure $\mu$ in \eqref{quadruple} near the origin (see \cite{GLT2017Limit} for details). In the infinite variance case, the limiting behavior is even more complex as the limit process may additionally depend on the regular variation index of the marginal distribution (see \cite{GLT2018LimitInfVar} for details). The limiting behavior of the integrated process has practical significance since supOU processes may be used as stochastic volatility models, see \cite{barndorff1997processes,barndorff2001non} and the references therein. In this setting the integrated process $X^*$ represents the integrated volatility (see e.g.~\cite{barndorff2013multivariate}). Moreover, the limiting behavior is important for statistical estimation (see \cite{nguyen2018bridging,stelzer2015moment}).

The integrated supOU process may exhibit another interesting limiting property related to behavior of their absolute moments in time. Although a suitably normalized integrated process satisfies a limit theorem, it may happen that its moments do not converge beyond some critical order. One way to investigate this behavior is to measure the rate of growth of moments by the \textit{scaling function}, defined for a generic process $Y=\{Y(t),\, t \geq 0\}$  as
\begin{equation}\label{deftauY}
\tau_Y(q) = \lim_{t\to \infty} \frac{\log \E |Y(t)|^q}{\log t},
\end{equation}
assuming the limit in \eqref{deftauY} exists and is finite. We will often focus on
\begin{equation*}
\frac{\tau_Y(q)}{q} = \lim_{t\to \infty} \frac{\log \left(\E |Y(t)|^q\right)^{1/q}}{\log t}
\end{equation*}
which has the advantage of involving $\left(\E |Y(t)|^q\right)^{1/q}$ which has the same units as $Y(t)$. The values $q$ are assumed to be in the range of finite moments $q \in (0,\overline{q}(Y))$, where
\begin{equation*}
\overline{q}(Y) = \sup \{ q >0 :\E|Y(t)|^q < \infty  \ \forall t\}.
\end{equation*}
To see how this is related to limit theorems, suppose that $Y$ satisfies a limit theorem in the form
\begin{equation*}
\left\{ \frac{Y(Tt)}{A_T} \right\} \overset{d}{\to} \left\{ Z(t) \right\},
\end{equation*}
with $A_T$ a sequence of constants and convergence in the sense of convergence of all finite-dimensional distributions as $T \to \infty$. By Lamperti's theorem (see, for example, \cite[Theorem 2.8.5]{pipiras2017long}), the limit $Z$ is $H$-self-similar for some $H>0$, that is, for any constant $c>0$, the finite-dimensional distributions of $Z(ct)$ are the same as those of $c^H Z(t)$. Moreover, the normalizing sequence is of the form $A_T=\ell(T) T^H$ for some $\ell$ slowly varying at infinity. For self-similar process, the moments evolve as a power function of time since $\E|Z(t)|^q=\E|Z(1)|^q t^{Hq}$ and therefore the scaling function of $Z$ is $\tau_Z(q)=Hq$. If for some $q>0$ we have
\begin{equation}\label{limitformmom}
\frac{\E| Y(Tt)|^q}{A_T^q} \to \E |Z(t)|^q, \quad \forall t \geq 0,
\end{equation}
then the scaling function of $Y$ would also be $\tau_Y(q)=Hq$ (see \cite[Theorem 1]{GLST2017Arxiv}), and the function
\begin{equation}\label{qeq}
 q \mapsto \frac{\tau_Y(q)}{q} = \frac{Hq}{q} = H
 \end{equation}
 would be constant over values of  $q$ for which \eqref{limitformmom} holds.
 
It was shown in \cite{GLST2017Arxiv} that the integrated supOU process $X^*$ may have the scaling function
\begin{equation}\label{tauX*}
\tau_{X^*}(q)=q-\alpha
\end{equation}
for a certain range of $q$. Thus its scaling function is different from that of a
 self-similar process.
 This situation happens, for example, for a non-Gaussian integrated supOU process with marginal distribution having exponentially decaying tails and probability measure $\pi$ in \eqref{quadruple} regularly varying at zero. 
 
  Note that the relation (\ref{tauX*}) implies that the function
\begin{equation*}
q \mapsto \frac{\tau_{X^*}(q)}{q} = \frac{q-\alpha}{q} = 1-\frac{\alpha}{q}
\end{equation*}
is not constant. It has points of strict increase, a property referred to as \textit{intermittency}.
 This term is used in all kind of different contexts. It refers in general to an unusual moment behavior and is used in various applications such as  turbulence, magnetohydrodynamics, rain and cloud studies, physics of fusion plasmas (see e.g.~\cite[Chapter 8]{frisch1995turbulence} or \cite{zel1987intermittency}).

Hence, intermittency implies that the usual convergence of moments \eqref{limitformmom} must not hold beyond some critical value of $q$. The papers \cite{GLST2017Arxiv,GLST2016JSP,GLT2017Limit} provide a complete picture on the behavior of moments in the case where $X(t)$ has {\em finite variance} .


We focus hereon the limiting behavior of moments and on the intermittency in the case where $X(t)$ has infinite variance and show that we can have intermittency even in this case. To establish the rate of growth of moments we make use of the limit theorems established in \cite{GLT2018LimitInfVar}. The type of the limiting process depends heavily on the structure of the underlying supOU process. Hence, the form of the scaling function of the integrated process will depend on the several parameters related to the quadruple \eqref{quadruple}. Special care is needed since the range of finite moments is limited. We show that the scaling function may look like a broken line indicating that there is a change-point in the rate of growth of moments. Hence, infinite variance integrated supOU processes may also exhibit the phenomenon of intermittency. Our results also indicate that in some cases, if we decompose the process into several components, the intermittency of the finite variance component may remain hidden by the infinite moments of the infinite variance component. We conclude that moments may have limited capability in identifying unusual limiting behavior.

The paper is organized as follows. In Section \ref{s:pre} we introduce notation and assumptions. Section \ref{s:main} contains the main results and all the proofs are given in Section \ref{s:proofs}.

\section{Preliminaries and assumptions}\label{s:pre}
We shall use the notation
\begin{equation*}
\kappa_Y(\zeta)=C\left\{ \zeta \ddagger Y\right\} = \log \E e^{i \zeta Y}, \quad \zeta \in \R,
\end{equation*}
to denote the cumulant (generating) function of a random variable $Y$. For a stochastic process $Y=\{Y(t)\}$ we write $\kappa_Y(\zeta,t) = \kappa_{Y(t)}(\zeta)$, and by suppressing $t$ we mean $\kappa_Y(\zeta)=\kappa_Y(\zeta,1)$, that is the cumulant function of the random variable $Y(1)$.

\subsection{SupOU processes}
The class of supOU processes has been introduced by Barndorff-Nielsen in \cite{bn2001} as follows. Let $m$ be the product $m=\pi \times Leb$ of a probability measure $\pi$ on $\R_+$ and the Lebesgue measure on $\R$. A homogeneous infinitely divisible random measure (\textit{L\'evy basis}) on $\R_+ \times \R$ with \textit{control measure} $m$ is a random measure such that the cumulant function of the random variable $\Lambda(A)$, where $A \in \mathcal{B} \left(\R_+ \times \R\right)$ has finite measure, equals
\begin{equation*}
C\left\{ \zeta \ddagger \Lambda(A)\right\} =  m(A) \kappa_{L}(\zeta) = \left( \pi \times Leb \right) (A) \kappa_{L}(\zeta).
\end{equation*}
Here $\kappa_{L}$ is the cumulant function $\kappa_{L} (\zeta)= \log \E e^{i \zeta L(1) }$ of some infinitely divisible random variable $L(1)$ with L\'evy-Khintchine triplet $(a,b,\mu)$ i.e.
\begin{equation}\label{kappacumfun}
\kappa_{L}(\zeta) = i\zeta a -\frac{\zeta ^{2}}{2} b  +\int_{\R}\left( e^{i\zeta x}-1-i\zeta x \mathbf{1}_{[-1,1]}(x)\right) \mu(dx).
\end{equation}
The L\'evy process $L=\{L(t), \, t\geq 0\}$ associated with the triplet $(a,b,\mu)$ is called the \textit{background driving L\'{e}vy process} (see \cite{barndorff2001non}). It has independent stationary increments and thus, its finite-dimensional distributions depend only on the distribution of $L(1)$.

The \textit{supOU} process is a strictly stationary process $X=\{X(t), \, t\in \R\}$ given by the stochastic integral (\cite{bn2001})
\begin{equation}\label{supOU}
X(t)= \int_{\xi=0}^\infty \int_{s=-\infty}^\infty e^{-\xi t + s} \mathbf{1}_{[0,\infty)}(\xi t -s) \Lambda(d\xi,ds).
\end{equation}
By appropriately choosing the infinitely divisible distribution $L(1)$, one can obtain any self-decomposable distribution as a marginal distribution of $X$. Note that the one-dimensional marginals of the supOU process are independent on the choice of $\pi$. The probability measure $\pi$ ``randomizes'' the rate parameter $\xi$ in \eqref{supOU} and the Lebesgue measure $ds$ is associated with the moving average variable $s$. The quadruple $(a,b,\mu,\pi)$ given in \eqref{quadruple} determines the law of the supOU process $\{X(t), \, t\in \R\}$. More details about supOU processes can be found in \cite{bn2001,barndorff2018ambit,barndorff2005spectral,barndorff2013levy,barndorff2011multivariate,GLST2017Arxiv}.

We will consider below supOU processes with marginal distributions in the domain of attraction of stable law. Recall that a stable distribution $\mathcal{S}_\gamma (\sigma, \rho, c)$ with parameters $0<\gamma<2$, $\sigma>0$, $-1\leq \rho \leq 1$ and $c\in \R$, has a cumulant function of the form:
\begin{equation}\label{cum:stable}
\kappa_{\mathcal{S}_\gamma (\sigma, \rho, c)}(\zeta) := C \{ \zeta \ddagger Z \} = i c \zeta - \sigma^{\gamma} |\zeta|^\gamma  \left( 1- i \rho \sign (\zeta) \chi (\zeta, \gamma) \right), \quad \zeta \in \R,
\end{equation}
where
\begin{equation*}
\chi (\zeta, \gamma) = \begin{cases}
\tan \left(\frac{\pi \gamma}{2}\right), & \gamma\neq 1,\\
\frac{\pi}{2} \log |\zeta|, & \gamma = 1.
\end{cases}
\end{equation*}
When $\gamma\neq 1$, then $\mathcal{S}_\gamma (\sigma, \rho, c)$ is strictly stable if and only if $c=0$. For $\gamma=1$, $\mathcal{S}_1 (\sigma, \rho, c)$ is strictly stable if and only if $\rho=0$.

\subsection{Basic assumptions}
We now state a set of assumptions for the class of supOU processes we consider. 

\begin{assumption}\label{assum}
The supOU process $\{X(t), \, t\in \R\}$ is such that the following holds:
\begin{enumerate}[(i)]
\item The marginal distribution satisfies 
\begin{equation}\label{regvarofX}
P(X(1)>x) \sim p k(x) x^{-\gamma} \quad \text{and} \quad P(X(1)\leq - x) \sim q k(x) x^{-\gamma}, \quad  \text{ as } x\to \infty,
\end{equation}
for some $p,q \geq 0$, $p+q>0$, $0<\gamma<2$ and some slowly varying function $k$ If $\gamma=1$, we assume $p=q$. When the mean is finite, we assume $\E X(1)=0$. 
\item $\pi$ has a density $p$ satisfying 
\begin{equation}\label{regvarofp}
p (x) \sim \alpha \ell(x^{-1}) x^{\alpha-1}, \quad \text{ as } x \to 0.
\end{equation}
for some $\alpha>0$ and some slowly varying function $\ell$ and
\begin{equation}\label{pifinitemean}
\int_0^\infty \xi \pi(d\xi)<\infty.
\end{equation}
\item The behavior at the origin of the L\'evy measure $\mu$ is given by
\begin{equation}\label{LevyMCond}
\mu \left( [x, \infty) \right) \sim c^+ x^{-\beta} \ \text{ and } \ \mu \left( (-\infty, -x] \right) \sim c^- x^{-\beta} \ \text{  as } x \to 0,
\end{equation}
for some $0\leq \beta<2$, $\beta\neq1+\alpha$, $c^+, c^- \geq 0$, $c^++c^->0$.
\end{enumerate}
\end{assumption}

Assumption \ref{assum}(i) implies that the marginal distribution is in the domain of attraction of an infinite variance stable law $\mathcal{S}_\gamma (\sigma, \rho, 0)$ with (see \cite[Theorem 2.6.1]{ibragimov1971independent})
\begin{equation}\label{sigmaandrho}
\sigma = \left( \frac{\Gamma(2-\gamma)}{1-\gamma} (p+q)  \cos \left(\frac{\pi \gamma}{2}\right) \right)^{1/\gamma}, \qquad \rho = \frac{p-q}{p+q}.
\end{equation}
Note that this is a strictly stable law since $\rho=0$ if $\gamma=1$. By \cite[Propositon 3.1]{fasen2007extremes}, the tail of the distribution function of $X(1)$ is asymptotically equivalent to the tail of the background driving L\'evy process $L(t)$ at $t=1$. More precisely, as $ x\to \infty$
\begin{equation}\label{equivalence of tails}
P(L(1)>x) \sim \gamma P(X(1)>x) \ \text{ and } \ P(L(1)\leq - x) \sim \gamma P(X(1)\leq -x).
\end{equation}
Hence, \eqref{regvarofX} implies
\begin{equation}\label{regvarofL}
P(L(1)>x) \sim p \gamma k(x) x^{-\gamma} \ \text{ and } \ P(L(1)\leq - x) \sim q \gamma k(x) x^{-\gamma}, \quad  \text{as } x\to \infty,
\end{equation}
and $L(1)$ is in the domain of attraction of stable distribution  $\mathcal{S}_\gamma (\gamma^{1/\gamma} \sigma, \rho, 0)$.

The next assumption, Assumption \ref{assum}(ii), concerns the dependence structure controlled by the behavior near the origin of the probability measure $\pi$ in the characteristic quadruple \eqref{quadruple}. In the finite variance case, $\pi$ is directly related to the correlation function of the supOU process $X$:
\begin{equation*}
r(t)=\int_{\R_+} e^{-t \xi }\pi (d\xi ), \quad t \geq 0.
\end{equation*}
Hence, by a Tauberian argument, the decay of the correlation function at infinity is related to the decay of the distribution function of $\pi$ at zero (see \cite[Proposition 2.6]{fasen2007extremes}). We assume $\pi$ has a density for simplicity. Note that if the variance of the supOU process is finite and $\alpha \in (0,1)$, then the correlation function is not integrable, and the finite variance supOU process may be said to exhibit long-range dependence. On the other hand, note that the tail distribution of $\pi$ does not affect the tail behavior of $r(t)$, and in particular the decay of correlations. Hence it is not very restrictive to assume that \eqref{pifinitemean} holds.

In Assumption \ref{assum}(iii), the L\'evy measure $\mu$ is assumed to have a power law behavior near the origin which will give rise to another parameter affecting the limiting behavior. We have excluded a boundary cases $\gamma=1+\alpha$ to simplify the presentation of the results. If \eqref{LevyMCond} holds, then $\beta$ is the Blumenthal-Getoor index of the L\'evy measure $\mu$ defined by (see \cite{GLT2017Limit})
\begin{equation*}
\beta_{BG} = \inf \left\{\gamma \geq 0 : \int_{|x|\leq 1} |x|^\gamma \mu(dx) < \infty \right\}.
\end{equation*}
Note that by \cite[Lemma 7.15]{kyprianou2014fluctuations}, $\mu \left( [x, \infty) \right) \sim P(L(1)>x)$ and $\mu \left( (-\infty, -x] \right) \sim P(L(1)\leq -x)$ as $x \to \infty$, hence we can express \eqref{regvarofL} equivalently as
\begin{equation}\label{e:alternativeDOAcondition}
\mu \left( [x, \infty) \right) \sim p \gamma k(x) x^{-\gamma} \ \text{ and } \ \mu \left( (-\infty, -x] \right) \sim q \gamma k(x) x^{-\gamma}, \quad  \text{ as } x\to \infty.
\end{equation}
Hence, all the assumptions can be stated in terms of the characteristic quadruple \eqref{quadruple}. The condition \eqref{LevyMCond} may be equivalently stated in terms of the L\'evy measure of $X(1)$. Indeed, if $\nu$ is the L\'evy measure of $X(1)$, then \eqref{LevyMCond} is equivalent to (see \cite{GLT2017Limit} for details)
\begin{equation*}
\nu \left( [x, \infty) \right) \sim \beta^{-1} c^+ x^{-\beta} \ \text{ and } \ \nu \left( (-\infty, -x] \right) \sim \beta^{-1}  c^- x^{-\beta} \ \text{  as } x \to 0.
\end{equation*}

\section{Main results}\label{s:main}
As stated in the introduction, we are interested in establishing the rate of growth of moments of the integrated process \eqref{integratedsupOU}, measured by the scaling function $\tau_{X^*}$ defined by \eqref{deftauY}. We particularly focus on whether the scaling function exhibits non-linearities. The situation is more delicate than in the finite variance case since the range of finite moments is limited and the scaling function of the integrated process $X^*$ is well-defined only over the interval $(0,\overline{q}(X^*))=(0,\gamma)$.

We will show that infinite variance supOU processes may exhibit the phenomenon of intermittency. We first consider the case when the underlying supOU process has no Gaussian component ($b=0$). The obtained scaling functions for this case are shown in Figures \ref{figalla}-\ref{figalld}.

\begin{theorem}\label{thm:mainb=0}
Suppose that Assumption \ref{assum} holds and $b=0$. Then the scaling function $\tau_{X^*}(q)$ of the process $X^*$ is as follows:
\begin{enumerate}[(a)]
	\item If $\alpha>1$ or if $\alpha \in (0,1)$ and $\gamma<1+\alpha$, then
	\begin{equation*}
	\tau_{X^*}(q)=\frac{1}{\gamma} q, \quad 0<q<\gamma.
	\end{equation*}
	\item If $\beta<1+\alpha<\gamma$, then
	\begin{equation*}
	\tau_{X^*}(q) = \begin{cases}
	\frac{1}{1+\alpha} q, & 0<q\leq 1+\alpha,\\
	q-\alpha, & 1+\alpha \leq q < \gamma.
	\end{cases}
	\end{equation*}
	\item If $1+\alpha<\beta\leq\gamma$, then
	\begin{equation*}
	\tau_{X^*}(q) = \begin{cases}
	\left(1-\frac{\alpha}{\beta} \right) q, & 0<q\leq \beta,\\
	q-\alpha, & \beta \leq q < \gamma.
	\end{cases}
	\end{equation*}
	\item If $1+\alpha<\gamma<\beta$, then
	\begin{equation*}
	\tau_{X^*}(q)=\left(1-\frac{\alpha}{\beta} \right) q, \quad 0<q<\gamma.
	\end{equation*}
\end{enumerate}
\end{theorem}

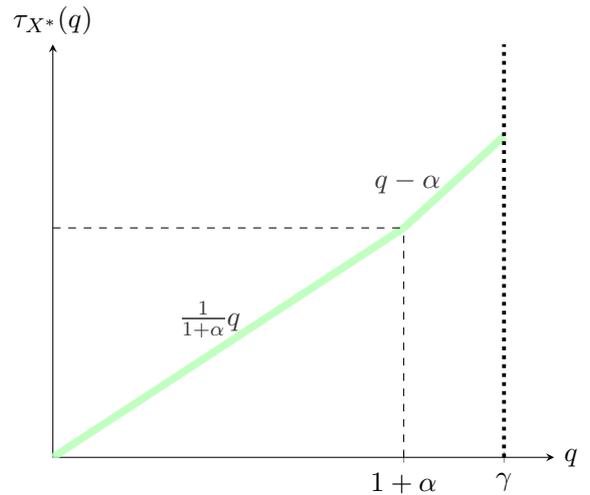
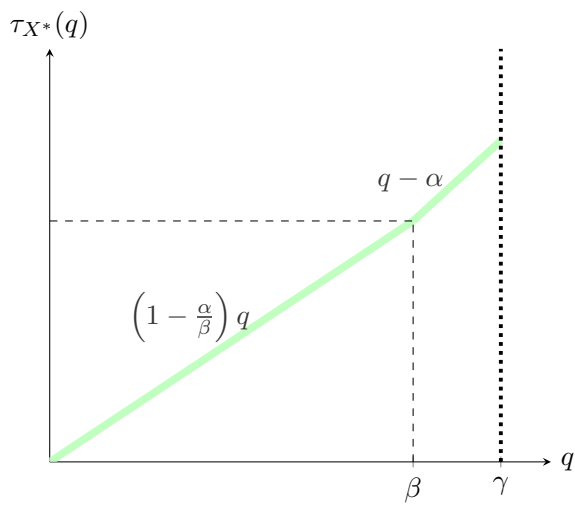
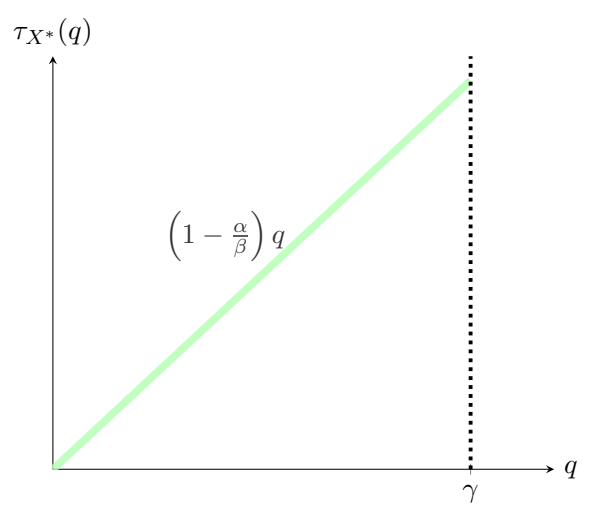
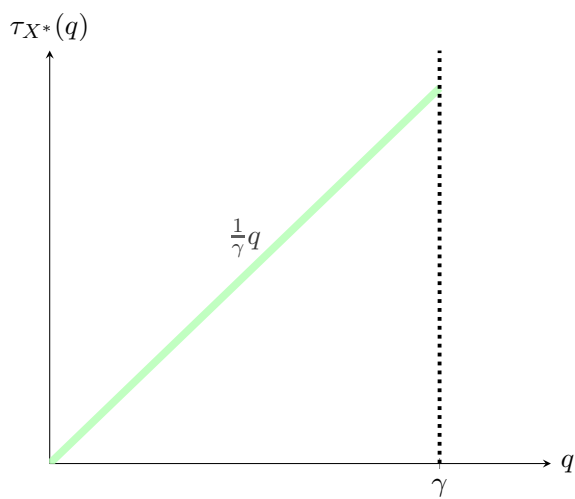
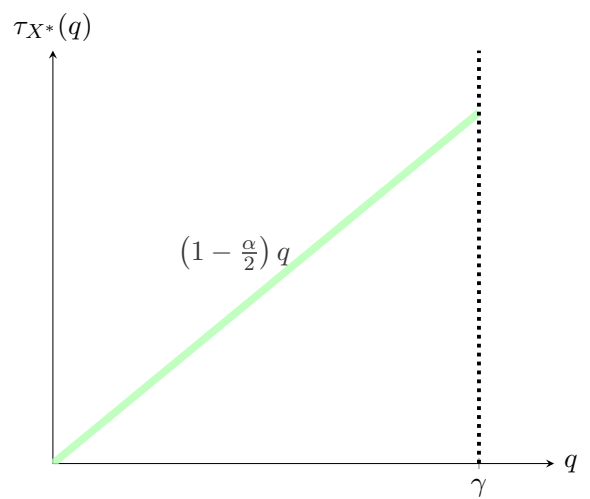
\begin{figure}
\centering
\begin{subfigure}[b]{0.45\textwidth}
\resizebox{1.1\textwidth}{!}{
\begin{tikzpicture}[domain=0:2]
\begin{axis}[
axis lines=middle,
xlabel=$q$, xlabel style={at=(current axis.right of origin), anchor=west},
ylabel=$\tau_{X^*}(q)$, ylabel style={at=(current axis.above origin), anchor=south},
xtick={0,1.5},
xticklabels={$0$,$\gamma$},
xmin=0,
xmax=1.8,
ymajorticks=false
]
\addplot[line width=3pt,opacity=0.8,white!70!green,domain=0:1.5]{x/1.5} node [pos=0.6,left,black]{$\frac{1}{\gamma}q\ $};
\addplot[dotted,line width=1.5pt] coordinates {(1.5,0) (1.5,1.05)};
\end{axis}
\end{tikzpicture}
}
\caption{Theorem \ref{thm:mainb=0}(a)}
\label{figalla}
\end{subfigure}
\hfill
\begin{subfigure}[b]{0.45\textwidth}
\resizebox{1.1\textwidth}{!}{
\begin{tikzpicture}
\begin{axis}[
axis lines=middle,
xlabel=$q$, xlabel style={at=(current axis.right of origin), anchor=west},
ylabel=$\tau_{X^*}(q)$, ylabel style={at=(current axis.above origin), anchor=south},
xtick={0,1.4,1.8},
xticklabels={$0$,$1+\alpha$,$\gamma$},
xmin=0,
xmax=2,
ymajorticks=false
]
\addplot[line width=3pt,opacity=0.8,white!70!green,domain=0:1.4]{(1/1.4)*x} node [pos=0.6,left, black]{$\frac{1}{1+\alpha}q\ $};
\addplot[line width=3pt,opacity=0.8,white!70!green,domain=1.4:1.8]{x-0.4} node [pos=0.5,left, black]{$q-\alpha$};
\addplot[dashed] coordinates {(1.4,0) (1.4,1)};
\addplot[dashed] coordinates {(0,1) (1.4,1)};
\addplot[dotted,line width=1.5pt] coordinates {(1.8,0) (1.8,1.8)};
\end{axis}
\end{tikzpicture}
}
\caption{Theorem \ref{thm:mainb=0}(b)}
\label{figallb}
\end{subfigure}
\begin{subfigure}[b]{0.45\textwidth}
\resizebox{1.1\textwidth}{!}{
\begin{tikzpicture}
\begin{axis}[
axis lines=middle,
xlabel=$q$, xlabel style={at=(current axis.right of origin), anchor=west},
ylabel=$\tau_{X^*}(q)$, ylabel style={at=(current axis.above origin), anchor=south},
xtick={0,1.45,1.8},
xticklabels={$0$,$\beta$,$\gamma$},
xmin=0,
xmax=2,
ymajorticks=false
]
\addplot[line width=3pt,opacity=0.8,white!70!green,domain=0:1.45]{(1-0.4/1.45)*x} node [pos=0.6,left,black]{$\left(1-\frac{\alpha}{\beta}\right)q\, $};
\addplot[line width=3pt,opacity=0.8,white!70!green,domain=1.45:1.8]{x-0.4} node [pos=0.5,left,black]{$q-\alpha$};
\addplot[dashed] coordinates {(1.45,0) (1.45,1.05)};
\addplot[dashed] coordinates {(0,1.05) (1.45,1.05)};
\addplot[dotted,line width=1.5pt] coordinates {(1.8,0) (1.8,1.8)};
\end{axis}
\end{tikzpicture}
}
\caption{Theorem \ref{thm:mainb=0}(c)}
\label{figallc}
\end{subfigure}
\hfill
\begin{subfigure}[b]{0.45\textwidth}
\resizebox{1.1\textwidth}{!}{
\begin{tikzpicture}
\begin{axis}[
axis lines=middle,
xlabel=$q$, xlabel style={at=(current axis.right of origin), anchor=west},
ylabel=$\tau_{X^*}(q)$, ylabel style={at=(current axis.above origin), anchor=south},
xtick={0,1.5},
xticklabels={$0$,$\gamma$},
xmin=0,
xmax=1.8,
ymajorticks=false
]
\addplot[line width=3pt,opacity=0.8,white!70!green,domain=0:1.5]{(1-0.1/1.7)*x} node [pos=0.6,left,black]{$\left(1-\frac{\alpha}{\beta}\right)q\, $};
\addplot[dotted,line width=1.5pt] coordinates {(1.5,0) (1.5,1.5)};
\end{axis}
\end{tikzpicture}
}
\caption{Theorem \ref{thm:mainb=0}(d)}
\label{figalld}
\end{subfigure}
\begin{subfigure}[b]{0.45\textwidth}
\resizebox{1.1\textwidth}{!}{
\begin{tikzpicture}[domain=0:2]
\begin{axis}[
axis lines=middle,
xlabel=$q$, xlabel style={at=(current axis.right of origin), anchor=west},
ylabel=$\tau_{X^*}(q)$, ylabel style={at=(current axis.above origin), anchor=south},
xtick={0,1.4},
xticklabels={$0$,$\gamma$},
xmin=0,
xmax=1.8,
ymajorticks=false
]
\addplot[line width=3pt,opacity=0.8,white!70!green,domain=0:1.4]{x/1.4} node [pos=0.6,left,black]{$\frac{1}{\gamma}q\ $};
\addplot[dotted,line width=1.5pt] coordinates {(1.4,0) (1.4,1.1)};
\end{axis}
\end{tikzpicture}
}
\caption{Theorem \ref{thm:mainb!=0}(a)}
\label{figalle}
\end{subfigure}
\hfill
\begin{subfigure}[b]{0.45\textwidth}
\resizebox{1.1\textwidth}{!}{
\begin{tikzpicture}
\begin{axis}[
axis lines=middle,
xlabel=$q$, xlabel style={at=(current axis.right of origin), anchor=west},
ylabel=$\tau_{X^*}(q)$, ylabel style={at=(current axis.above origin), anchor=south},
xtick={0,1.7},
xticklabels={$0$,$\gamma$},
xmin=0,
xmax=2,
ymajorticks=false
]
\addplot[line width=3pt,opacity=0.8,white!70!green,domain=0:1.7]{(1-0.5/2)*x} node [pos=0.6,left,black]{$\left(1-\frac{\alpha}{2}\right)q\, $};
\addplot[dotted,line width=1.5pt] coordinates {(1.7,0) (1.7,1.5)};
\end{axis}
\end{tikzpicture}
}
\caption{Theorem \ref{thm:mainb!=0}(b)}
\label{figallf}
\end{subfigure}
\caption{The scaling functions obtained in Theorems \ref{thm:mainb=0} ($b=0$) and \ref{thm:mainb!=0} ($b\neq 0$). There is intermittency in the cases (b) and (c).}\label{figall}
\end{figure}

Note that the scaling function has a change-point in only two of the cases of Theorem \ref{thm:mainb=0}. Hence intermittency appears only in cases (b) and (c) of Theorem \ref{thm:mainb=0} shown in Figures \ref{figallb} and \ref{figallc}, respectively. One can notice that infinite order moments may hide the intermittency property as they limit the domain of the scaling function. 

The proof of Theorem \ref{thm:mainb=0} is given in Subsection \ref{s:sfX*}. It is based on the decomposition of the integrated process $X^*$ into independent components $X^*_1$, $X^*_2$ and $X_3^*$ that correspond to characteristic quadruples $(a,0,\mu(dx) \1_{\{|x|>1\}},\pi)$, $(0,0,\mu(dx) \1_{\{|x|\leq 1\}},\pi)$ and $(0,b,0,\pi)$, respectively. In Section \ref{s:proofs} we derive the scaling functions of $X^*_1$, $X^*_2$ and $X_3^*$ and then combine these to get the scaling function of the integrated process $X^*$. This is illustrated in Figure \ref{fig4} in Subsection \ref{s:sfX*}. 

The finite variance component $X_2^*$ exhibits intermittency in all cases, however, this is not always apparent from the scaling function of the process $X^*$. In these cases, the change point in the scaling function of $X_2^*$ is to the right of the moment index $\gamma$ and the scaling function of $X^*$ remains linear on $(0,\gamma)$ (see Figures \ref{fig4a}, \ref{fig4b}, \ref{fig4c} and \ref{fig4f} in Subsection \ref{s:sfX*}). Hence, infinite order moments may hide the behavior of the intermittent component.

We next state the result for the supOU process with Gaussian component ($b\neq 0$). The scaling functions for this case are shown in Figures \ref{figalle}-\ref{figallf}.

\begin{theorem}\label{thm:mainb!=0}
Suppose that Assumption \ref{assum} holds and $b\neq 0$. Then the scaling function $\tau_{X^*}(q)$ of the process $X^*$ is as follows:
\begin{enumerate}[(a)]
	\item If $\alpha>1$ or if $\alpha \in (0,1)$ and $\gamma <\frac{2}{2-\alpha}$, then
	\begin{equation*}
	\tau_{X^*}(q)=\frac{1}{\gamma} q, \quad 0<q<\gamma.
	\end{equation*}
	\item If $\alpha \in (0,1)$ and $\gamma > \frac{2}{2-\alpha}$, then
	\begin{equation*}
	\tau_{X^*}(q)=\left(1-\frac{\alpha}{2} \right) q, \quad 0<q<\gamma.
	\end{equation*}
\end{enumerate}
\end{theorem}

Note that if the Gaussian component is present, then the scaling function displays no intermittency. For example, even if the scaling functions of the two components $X_1^*$ and $X_2^*$ have a change-point, this cannot be seen from the scaling function of $X^*$ due to infinite moments (see Figures \ref{fig6c}, \ref{fig6d}, \ref{fig6e} in Subsection \ref{s:sfX*}).

\section{Proofs}\label{s:proofs}

For the proofs of the main results, we first make a decomposition of the integrated process $X^*$ into components that have different limiting behavior. We then compute the scaling functions of these components and finally combine them to get the scaling function of the integrated process.

\subsection{The basic decomposition}\label{s:decomposition}
The decomposition is based on the L\'evy-It\^{o} decomposition of the background driving L\'evy process $L$. Let
\begin{align*}
\mu_1(dx)&=\mu(dx) \1_{\{|x|>1\}},\\
\mu_{2}(dx)&=\mu(dx) \1_{\{|x|\leq1\}},
\end{align*}
where $\mu$ is the L\'evy measure of the L\'evy process $L$. Then we can make a decomposition of the L\'evy basis into independent components:
\begin{itemize}
\item $\Lambda_1$ with characteristic quadruple $(a,0,\mu_1,\pi)$,
\item $\Lambda_2$ with characteristic quadruple $(0,0,\mu_2,\pi)$,
\item $\Lambda_3$ with characteristic quadruple $(0,b,0,\pi)$.
\end{itemize}
Note that if $X(1)$ has finite mean, then the assumption $\E X(1)=0$ implies that $\E L(1)=0$ (see \cite[Eq.~(2.8)]{bn2001}) and we must have $a=-\int_{|x|>1} |x|\mu (dx)$ (see e.g.~\cite[Ex.~25.12]{sato1999levy}). Let $L_1(t)$, $L_2(t)$ and $L_3(t)$, $t\in \R$ denote the corresponding background driving L\'evy processes so that we have the following cumulant functions:
\begin{align}
C \left\{ \zeta \ddagger L_1(1) \right\} & = i\zeta a +\int_{\R}\left( e^{i\zeta x}-1\right) \mu_1(dx) = i\zeta a +\int_{|x|>1}\left( e^{i\zeta x}-1\right) \mu(dx),\label{kappaL1}\\
C \left\{ \zeta \ddagger L_2(1) \right\} &=\int_{\R}\left( e^{i\zeta x}-1-i\zeta x \mathbf{1}_{[-1,1]}(x)\right) \mu_2(dx)\nonumber\\
&=\int_{|x|\leq 1}\left( e^{i\zeta x}-1-i\zeta x\mathbf{1}_{[-1,1]}(x)\right) \mu(dx) ,\nonumber\\
C \left\{ \zeta \ddagger L_3(1) \right\} &= -\frac{\zeta ^{2}}{2} b.\nonumber
\end{align}
Note that $L_1$ is a compound Poisson process and $L_3$ is Brownian motion. Consequently, we can represent $X(t)$ as
\begin{equation}\label{e:decomposition}
\begin{aligned}
X(t) &= \int_{0}^\infty \int_{-\infty }^{\xi t}e^{-\xi t + s} \Lambda_1(d\xi,ds) + \int_{0}^\infty \int_{-\infty }^{\xi t}e^{-\xi t + s} \Lambda_2(d\xi,ds)\\
&\hspace{4cm}+ \int_{0}^\infty \int_{-\infty }^{\xi t}e^{-\xi t + s} \Lambda_3(d\xi,ds)\\
&=: X_1(t) + X_2(t) + X_3(t),
\end{aligned}
\end{equation}
with $X_1$, $X_2$ and $X_3$ independent. In the following, $X^*_1$, $X^*_2$ and $X_3^*$ will denote the corresponding integrated processes which are independent. \\

Before we proceed, we note here two technical facts that will be used in the proofs below. The first is a stochastic Fubini theorem related to the change of the order of integration for the integrated
process. It has been used implicitly in many references (see e.g.~\cite{bn2001,GLST2017Arxiv,GLT2017Limit}).

\begin{lemma}\label{lemma:Fubini}
For the integrated supOU process $X^*$ one has
\begin{equation}\label{e:changeofintorder}
X^*(t)=\int_0^t \left(\int_{\R_+\times \R} f(u, \xi, s) \Lambda (d\xi, ds) \right) du =  \int_{\R_+\times \R} \left( \int_0^t f(u, \xi, s) du \right) \Lambda (d\xi, ds), \ a.s.
\end{equation}
where $f(u, \xi, s) = e^{- \xi u + s} \1_{[0,\infty)} (\xi u -s)$.
\end{lemma}

\begin{proof}
If $\E |X(1)| < \infty$, then we can directly use a stochastic Fubini theorem given in \cite[Theorem 3.1]{barndorff2011quasi}. The conditions of Theorem 3.1 and Remark 3.2 in \cite[Theorem 3.1]{barndorff2011quasi} boil down to showing that
\begin{enumerate}[(i)]
\item for every $u\in [0,t]$, $f(u, \cdot, \cdot)$ is in the Musielak-Orlicz space $L_{\phi_1}$, that is
\begin{equation*}
\int_{0}^\infty \int_{-\infty}^\infty \left(\sigma^2 f(u, \xi, s)^2 + \int_{\R} (|x f(u, \xi ,s)|^2 \wedge |x f(u, \xi, s)|) \mu(dx) \right) \pi(d\xi) ds < \infty,
\end{equation*}
\item it holds that
\begin{equation*}
\int_0^t \int_{0}^\infty \int_{-\infty}^\infty \left(\sigma^2 f(u, \xi, s)^2 + \int_{\R} (|x f(u, \xi ,s)|^2 \wedge |x f(u, \xi, s)|) \mu(dx) \right) \pi(d\xi) ds du < \infty.
\end{equation*}
\end{enumerate} 
By \cite[Theorem 3.3]{rajput1989spectral}, $L_{\phi_1}$ coincides with the space of $\Lambda$-integrable functions $g$ such that $\E |\int g d\Lambda|<\infty$. Theorem 3.1 of \cite{bn2001} shows that $f(u,\cdot,\cdot)$ is $\Lambda$-integrable and since we have assumed $\E |X(u)|<\infty$, we conclude that condition (i) holds. By the change of variables $r= e^{- \xi u + s}$ we get
\begin{align*}
\int_0^t \int_{\R_+\times \R} &\left(\sigma^2 e^{2(- \xi u + s)} \1_{[0,\infty)} (\xi u -s) + \int_{\R} (|x e^{- \xi u + s}|^2 \wedge |x e^{- \xi u + s} |) \1_{[0,\infty)} (\xi u -s) \mu(dx) \right) \pi(d\xi) ds du\\
&=\int_0^t \int_{\R_+\times \R} \left(\sigma^2 r^2 \1_{(0,1]}(r) + \int_{\R} (|x r|^2 \wedge |x r|) \1_{(0,1]}(r) \mu(dx) \right) \pi(d\xi) r^{-1} dr du\\
&=t \int_{\R_+\times \R} \left(\sigma^2 r^2 \1_{(0,1]}(r) + \int_{\R} (|x r|^2 \wedge |x r|) \1_{(0,1]}(r) \mu(dx) \right) \pi(d\xi) r^{-1} dr\\
&=t \int_{0}^\infty \int_{-\infty}^\infty \left(\sigma^2 f(u, \xi, s)^2 + \int_{\R} (|x f(u, \xi ,s)|^2 \wedge |x f(u, \xi, s)|) \mu(dx) \right) \pi(d\xi) ds,
\end{align*}
hence, (ii) follows from (i).

Suppose now that $\E |X(1)| = \infty$. We can decompose the L\'evy basis similarly as in \eqref{e:decomposition} into independent L\'evy basis $\Lambda_1'$ with characteristic quadruple $(0,0,\mu_1,\pi)$, $\mu_1(dx)=\mu(dx) \1_{\{|x|>1\}}$, and $\Lambda_2'$ with characteristic quadruple $(a,b,\mu_2,\pi)$, $\mu_2(dx)=\mu(dx) \1_{\{|x|\leq 1\}}$. For the integral with respect to $\Lambda_2$ we can apply \cite[Theorem 3.1]{barndorff2011quasi} as in the previous case. It remains to consider $\Lambda_1$, which is a compound Poisson random measure and can be written as
\begin{equation*}
\Lambda_1(A)=\int_A \int_\R  x N(dw, dx),
\end{equation*} 
where $N$ is a Poisson random measure on $\R_+\times \R\times \R$ with intensity $\pi \times Leb \times \mu_1$. We can represent $\Lambda_1(A)$ as
\begin{equation*}
\Lambda_1(A) = \sum_{k=-\infty}^{\infty} Z_k \delta_{(R_k, \Gamma_k)} (A),
\end{equation*}
where $-\infty<\cdots<\Gamma_{-1} < \Gamma_0 \leq 0 < \Gamma_1 < \cdots<\infty$ are the jump times of a Poisson process on $\R$ with intensity $\mu_1(\R)$, $\{Z_k, \, k \in \Z\}$ is an i.i.d.~sequence with distribution $\mu_1(dx)/\mu_1(\R)$, $\{R_k, \, k \in \Z\}$ is an i.i.d.~sequence with distribution $\pi$ and all three sequences are independent (see e.g.~\cite{fasen2007extremes}). The supOU process can then be represented as
\begin{align*}
X(u) &= \sum_{k=-\infty}^{\infty} Z_k  e^{- R_k u + \Gamma_k} \1_{[0,\infty)} (R_k u - \Gamma_k)\\
&= \sum_{k=-\infty}^{0} Z_k  e^{- R_k u + \Gamma_k}  + \sum_{k=1}^{\infty} Z_k  e^{- R_k u + \Gamma_k} \1_{[0,\infty)} (R_k u - \Gamma_k).
\end{align*}
The second sum has finitely many terms a.s.~due to $\1_{[0,\infty)} (R_k u - \Gamma_k)$ term, hence one can change the order of integration when integrating with respect to $u$. For the first sum, we have by using the inequality $(1-e^{-x})/x\leq 1$, $x>0$,
\begin{equation*}
\sum_{k=-\infty}^{0} \int_0^t |Z_k|  e^{- R_k u + \Gamma_k} du = \sum_{k=-\infty}^{0} |Z_k| e^{\Gamma_k} R_k^{-1} (1-e^{-R_k t}) \leq t \sum_{k=-\infty}^{0} |Z_k| e^{\Gamma_k}.
\end{equation*}
The right-hand side is finite since it is the integral of $e^{-x}$ with respect to compound Poisson random measure with intensity $Leb\times |\mu_1|$ (see e.g.~\cite{last2017lectures}). By the classical Fubini-Tonelli theorem we can change the order of integration. This completes the proof of \eqref{e:changeofintorder}.
\end{proof}

The second fact concerns again $X^*(t)$ in \eqref{e:changeofintorder}. Clearly $\E |X^*(t)|^q<\infty$ for $q<\gamma$. The next lemma shows that $\E |X^*(t)|^q=\infty$ for $q>\gamma$.

\begin{lemma}\label{lemma:momentsX*}
If the supOU process $X$ satisfies \eqref{regvarofX} for some $\gamma>0$, then for the integrated process $X$ we have $\E |X^*(t)|^q=\infty$ for $q>\gamma$ and every $t>0$.
\end{lemma}

\begin{proof}
We will show that $\E |X^*(t)|^{\gamma+\varepsilon}=\infty$ for $\varepsilon>0$. By \eqref{e:changeofintorder}, $X^*(t)$ is representable as an integral with respect to L\'evy basis $\Lambda$
\begin{equation*}
X^*(t) = \int_{\R_+\times \R} g_t(\xi, s) \Lambda(d\xi, ds).
\end{equation*}
where
\begin{equation*}
g_t (\xi, s) = \int_0^t e^{- \xi u + s} \1_{[0,\infty)} (\xi u -s) du = \begin{cases}
e^s \xi^{-1} (1-e^{-\xi t}), & s<0,\\
\xi^{-1} (1-e^{-\xi t+s}), & 0<s<\xi t.\\
\end{cases}
\end{equation*}
Hence, by \cite[Theorem 2.7]{rajput1989spectral}, the distribution of $X^*(t)$ is infinitely divisible and for Borel set $B\subseteq\R$, the L\'evy measure $\mu_{g_t}$ of $X^*(t)$ is given by
\begin{equation*}
\mu_{g_t} (B)= \pi\times Leb \times \mu \left(\{(\xi,s,x) : g_t(\xi,s) x \in B\setminus \{0\}\} \right).
\end{equation*}
In particular, for $y>0$ and $B=[y,\infty)$, we have
\begin{equation*}
\mu_{g_t}\left( [y, \infty) \right) = \int_{0}^\infty \int_{-\infty}^{\infty} \mu \left([y/g_t(\xi,s)), \infty \right) \pi(d\xi) ds.
\end{equation*}
From \eqref{e:alternativeDOAcondition}, which is equivalent to \eqref{regvarofX}, one has for any $\delta<\varepsilon$, a $y_0$ such that $\mu \left([y,\infty)\right) \geq y^{-\gamma-\delta}$ for $y\geq y_0$. This implies that for $y\geq y_0$, $\mu_{g_t}\left( [y, \infty) \right) \geq C y^{-\gamma-\delta}$, where $C=\int_{0}^\infty \int_{-\infty}^{\infty} (g_t(\xi,s)))^{\gamma+\delta} \pi(d\xi) ds$. The same argument can be used for $\mu(-\infty, -y])$. But this implies that $\int_{\{|y|> 1\}} |y|^{\gamma+\varepsilon} \mu_{g_t}(dy)\geq C_1 + C_2 \int_{\{|y|\geq y_0\}} |y|^{\varepsilon-\delta}\mu_{g_t}(dy)=\infty$, where $C_1$ and $C_2$ are positive constants. Hence, we have $\E |X^*(t)|^{\gamma + \varepsilon}= \infty$ (see e.g.~\cite[Theorem 25.3]{sato1999levy}). 
\end{proof}

\subsection{Evaluation of the three scaling functions}\label{s:3sf}
We next investigate the scaling functions of each process $X^*_1$, $X^*_2$ and $X_3^*$ separately. These results will then be combined to give the scaling function of the integrated process.

\subsubsection{The scaling function of $X_1^*$}
The process $X_1^*$ has infinite moments of order greater than $\gamma$ and its scaling function $\tau_{X_1^*}$ is well-defined for $q \in (0,\gamma)$ (see Lemma \ref{lemma:momentsX*}). Following \cite[Lemma 5.1 and 5.2]{GLT2018LimitInfVar}, two processes may arise as a limit of $X_1^*$ after normalization.

If $\gamma<1+\alpha$, then as $T\to \infty$
\begin{equation}\label{e:X1*limitSRD}
\left\{ \frac{1}{T^{1/\gamma} k^{\#}(T)^{1/\gamma}} X_1^*(Tt) \right\} \overset{d}{\to} \left\{L_{\gamma} (t) \right\},
\end{equation}
where $k$ is the slowly varying function in \eqref{regvarofX}, $k^{\#}$ is the de Bruijn conjugate of $1/k(x^{1/\gamma})$ and the limit $\{L_{\gamma}\}$ is a $\gamma$-stable L\'evy process such that $L_{\gamma}(1)\overset{d}{=} \mathcal{S}_\gamma (\widetilde{\sigma}_{1,\gamma}, \rho, 0)$ with
\begin{equation*}
\widetilde{\sigma}_{1,\gamma} = \sigma \left( \gamma \int_0^\infty \xi^{1-\gamma} \pi(d\xi) \right)^{1/\gamma},
\end{equation*}
and $\sigma$ and $\rho$ given by \eqref{sigmaandrho}. Recall that the de Bruijn conjugate \cite[Subsection 1.5.7]{bingham1989regular} of some slowly varying function $h$ is a slowly varying function $h^{\#}$ such that
\begin{equation*}
h(x) h^{\#} \left(x h(x) \right) \to 1, \qquad h^{\#}(x) h (x h^{\#}(x)) \to 1,
\end{equation*}
as $x\to \infty$. By \cite[Theorem 1.5.13]{bingham1989regular} such function always exists and is unique up to asymptotic equivalence.

If, on the other hand $\gamma>1+\alpha$, then as $T\to \infty$
\begin{equation}\label{e:X1*limitLRD}
\left\{ \frac{1}{T^{1/(1+\alpha)} \ell^{\#}\left(T \right)^{1/(1+\alpha)}} X_1^*(Tt) \right\} \overset{d}{\to} \left\{L_{1 + \alpha} (t) \right\},
\end{equation}
where $\ell^{\#}$ is de Bruijn conjugate of $1/\ell(x^{1/(1+\alpha)})$ and the limit $\{L_{1+\alpha}\}$ is $(1+\alpha)$-stable L\'evy process such that $L_{1+\alpha}(1)\overset{d}{=} \mathcal{S}_\gamma (\widetilde{\sigma}_{1,\alpha}, \widetilde{\rho}_1, 0)$ with
\begin{equation}\label{sigma1alpha}
\widetilde{\sigma}_{1,\alpha} = \left( \frac{\Gamma(1-\alpha)}{\alpha} (c^-_1+c^+_1)  \cos \left(\frac{\pi (1+\alpha)}{2}\right) \right)^{1/(1+\alpha)}, \qquad \widetilde{\rho_1} = \frac{c^-_1 - c^+_1}{c^-_1+c^+_1},
\end{equation}
and  $c^-_1, c^+_1$ given by
\begin{equation}\label{c-+1}
c^-_1 = \frac{\alpha}{1+\alpha} \int_{-\infty}^{-1} |y|^{1+\alpha} \mu(dy), \qquad c^+_1 = \frac{\alpha}{1+\alpha} \int_1^{\infty} y^{1+\alpha} \mu(dy).
\end{equation}

We now consider convergence of moments in these limit theorems. First, if $\gamma<1+\alpha$, then we get the following scaling function for the process $X_1^*$.

\begin{lemma}\label{lemma:X1:1}
If Assumption \ref{assum} holds and $\gamma<1+\alpha$, then
\begin{equation*}
\tau_{X_1^*}(q) = \frac{1}{\gamma} q, \quad  0<q<\gamma.
\end{equation*}
\end{lemma}

\begin{proof}
Let $q<\gamma$ and $A_T=T^{1/\gamma} k^{\#}(T)^{1/\gamma}$. We will show that $\{|A_T^{-1} X_1^*(Tt)|^q\}$ is uniformly integrable so that $\E |A_T^{-1} X_1^*(Tt)|^q \to \E |L_{\gamma}(t)|^q$ as $T\to\infty$, where $\{L_{\gamma}\}$ is as in \eqref{e:X1*limitSRD}.

First we recall some known results. If $Y$ is some random variable, let $\widetilde{Y}$ denote its symmetrization, i.e.~$\widetilde{Y}=Y-Y'$ with $Y'=^d Y$ and independent of $Y$. By \cite[Lemma 4]{von1965inequalities}, if $r \in [1,2]$, $\E|Y|^r<\infty$ and $\E Y=0$, then
\begin{equation}\label{symmetricbound:geq1}
\E|Y|^r \leq \E |\widetilde{Y}|^r.
\end{equation}
On the other hand, if $r <1$ and $\E|Y|^r<\infty$, then we obtain from \cite[Proposition 3.6.4]{gut2013probability} that
\begin{equation}\label{symmetricbound:leq1}
\E|Y|^r \leq 2\E |\widetilde{Y}|^r + 2 | \med (Y)|^r,
\end{equation}
where $\med(Y)$ denotes the median of $Y$. Furthermore, one may express $r$-th absolute moment, $0<r<2$ as \cite[Lemma 2]{von1965inequalities}
\begin{equation}\label{e:thm:vonbahr2}
\E |Y|^r = k_r \int_{-\infty}^{\infty} \left( 1- \Re \exp \kappa_Y(\zeta) \right) |\zeta|^{-r-1} d \zeta
\end{equation}
where $k_r>0$ is a constant.

Consider now the symmetrized random variable $\widetilde{X_1^*}(Tt)$. The characteristic function of $\widetilde{X_1^*}(Tt)$ is $|\exp \kappa_{X_1^*} (\zeta, Tt)|^2$, hence from \eqref{e:thm:vonbahr2} we get
\begin{equation}\label{momSLPcase111}
\E \left| A_T^{-1} \widetilde{X_1^*}(Tt)\right|^q = k_q \int_{-\infty}^{\infty} \left( 1- |\exp \kappa_{X_1^*} (A_T^{-1} \zeta, Tt)|^2 \right) |\zeta|^{-q-1} d \zeta.
\end{equation}
In order to bound the integral in \eqref{momSLPcase111}, we shall first derive the bounds for $|\kappa_{X_1^*} (A_T^{-1} \zeta, Tt)|$. For this we make the decomposition from \eqref{supOU} by using Lemma \ref{lemma:Fubini}:
\begin{equation}\label{e:thmSRD1:decomposition}
\begin{aligned}
&X_1^*(Tt) = \int_{u=0}^{T t} \int_{\xi=0}^\infty \int_{s=-\infty}^{\xi u} e^{-\xi u + s} \Lambda_1(d\xi,ds) du  \\
&=\int_{u=0}^{T t} \int_{\xi=0}^\infty \int_{s=-\infty}^{0} e^{-\xi u + s} du \Lambda_1(d\xi,ds) + \int_{u=0}^{T t} \int_{\xi=0}^\infty \int_{s=0}^{\xi u} e^{-\xi u + s} du \Lambda_1(d\xi,ds)\\
&=\int_{\xi=0}^\infty \int_{s=-\infty}^{0} \int_{u=0}^{T t} e^{-\xi u + s} du \Lambda_1(d\xi,ds) +  \int_{\xi=0}^\infty \int_{u=0}^{T t} \int_{s=0}^{\xi u} e^{-\xi u + s} du \Lambda_1(d\xi,ds)\\
&=\int_{\xi=0}^\infty \int_{s=-\infty}^{0} \int_{u=0}^{T t} e^{-\xi u + s} du \Lambda_1(d\xi,ds) +  \int_{\xi=0}^\infty \int_{s=0}^{\xi T t} \int_{u=s/\xi}^{T t} e^{-\xi u + s} du \Lambda_1(d\xi,ds)\\
&=: \Delta X^*_{1,1}(Tt) + \Delta X^*_{1,2}(Tt),
\end{aligned}
\end{equation}
where we have used the fact that
\begin{equation*}
\1_{\{0\leq u \leq Tt\}} \1_{\{0\leq s \leq \xi u\}} = \1_{\{0 \leq s/\xi \leq u \leq Tt\}} = \1_{\{0 \leq s \leq \xi Tt\}} \1_{\{s/\xi\leq u \leq Tt\}}.
\end{equation*}
Since $\Delta X^*_{1,1}(Tt)$ and $\Delta X^*_{1,2}(Tt)$ are independent, we get
\begin{equation}\label{eq:thmSRD1:twoparts}
| \kappa_{X_1^*} (A_T^{-1} \zeta, Tt) | \leq | \kappa_{\Delta X^*_{1,1}} (A_T^{-1} \zeta, Tt) | + |\kappa_{\Delta X^*_{1,2}} (A_T^{-1} \zeta, Tt)|.
\end{equation}
Now we consider bounds for each term separately.

\begin{itemize}[leftmargin=*]
\item For the first term on the right hand side we use some parts of the proof of \cite[Lemma 5.1]{GLT2018LimitInfVar}. From the integration formula for the stochastic integral, for any $\Lambda$-integrable function $f$ on $\R_+ \times \R$, one has (see \cite{rajput1989spectral})
\begin{equation}\label{integrationrule}
C\left\{ \zeta \ddagger \int_{\R_+ \times \R}f d\Lambda \right\} = \int_{\R_+ \times \R} \kappa_{L} (\zeta f(\xi,s)) ds \pi(d\xi)
\end{equation}
and we get that
\begin{align}
\kappa_{\Delta X^*_{1,1}} (A_T^{-1} \zeta, Tt) &= \int_0^\infty \int_{-\infty}^{0} \kappa_{L_1} \left( \zeta A_T^{-1} \int_{0}^{Tt} e^{-\xi u + s} du \right) ds \pi(d\xi)\nonumber \\
&= \int_0^\infty \int_{-\infty}^{0} \kappa_{L_1} \left( \zeta A_T^{-1} e^s \xi^{-1} \left(1-e^{-\xi Tt} \right)\right) ds \pi(d\xi).\label{eq:proof:SRD1X1}
\end{align}
The assumption \eqref{regvarofL}, together with \cite[Theorem 2.6.4]{ibragimov1971independent}, imply that
\begin{equation}\label{eq:proof:SRD1X1:ibragimov}
\kappa_{L_1} (\zeta) \sim  k(1/|\zeta|) \kappa_{\mathcal{S}_\gamma (\gamma^{1/\gamma} \sigma, \rho, 0)}(\zeta), \quad \text{ as } \zeta \to 0.
\end{equation}
Since $|\kappa_{\mathcal{S}_\gamma (\gamma^{1/\gamma} \sigma, \rho, 0)}(\zeta)| = C |\zeta|^\gamma$ and $k$ is slowly varying at infinity, then, for arbitrary $\delta>0$, in some neighborhood of the origin one has
\begin{equation*}
|\kappa_{L_1}(\zeta) | \leq C_1 |\zeta|^{\gamma-\delta}, \quad |\zeta| \leq \varepsilon.
\end{equation*}
On the other hand, since $\left| e^{i\zeta x}-1 \right| \leq 2$, we have from \eqref{kappaL1} that
\begin{equation*}
|\kappa_{L_1}(\zeta) | \leq |a| |\zeta| +  2 \int_{\R} \mathbf{1}_{\{|x|> 1\}} \mu(dx) \leq |a| |\zeta| + C_2,
\end{equation*}
since the L\'evy measure is integrable on $\{|x|> 1\}$. By taking $C_3$ large enough we arrive at the bound
\begin{equation}\label{eq:proof:SRD1X1:bound}
|\kappa_{L_1}(\zeta) | \leq C_1 |\zeta|^{\gamma-\delta} \1_{\{|\zeta| \leq \varepsilon\}}  + C_3 |\zeta| \1_{\{|\zeta| > \varepsilon\}}.
\end{equation}
Now we have from \eqref{eq:proof:SRD1X1}
\begin{align}
&\left| \kappa_{\Delta X^*_{1,1}} (A_T^{-1} \zeta, Tt) \right| \nonumber\\
&\leq C_1 \int_0^\infty \int_{-\infty}^{0} \left|\zeta A_T^{-1}  e^s \xi^{-1} \left(1-e^{-\xi Tt} \right)\right|^{\gamma-\delta} \1_{\{|\zeta A_T^{-1}  e^s \xi^{-1} \left(1-e^{-\xi Tt} \right)| \leq \varepsilon\}}  ds \pi(d\xi)\nonumber\\
&\ + C_3 \int_0^\infty \int_{-\infty}^{0} \left|\zeta A_T^{-1}  e^s \xi^{-1} \left(1-e^{-\xi Tt} \right)\right| \1_{\{|\zeta A_T^{-1}  e^s \xi^{-1} \left(1-e^{-\xi Tt} \right)| > \varepsilon\}}  ds \pi(d\xi)\nonumber\\
&\leq C_1 |\zeta|^{\gamma-\delta} A_T^{-\gamma+\delta} \int_0^\infty \int_{-\infty}^{0} e^{(\gamma-\delta)s} \left(\xi^{-1} \left(1-e^{-\xi Tt} \right)\right)^{\gamma-\delta} ds \pi(d\xi)\nonumber\\
& \ + C_3 |\zeta| t A_T^{-1} T \int_0^\infty \int_{-\infty}^{0} e^{s} (\xi Tt)^{-1} \left(1-e^{-\xi Tt} \right) \1_{\{|\zeta A_T^{-1} \xi^{-1} \left(1-e^{-\xi Tt} \right)| > \varepsilon\}} ds \pi(d\xi)\nonumber\\
& \leq C_1 \frac{1}{\gamma - \delta} |\zeta|^{\gamma-\delta} t^{\gamma-\delta} A_T^{-\gamma+\delta} T^{\gamma-\delta}\int_0^\infty \left((\xi Tt)^{-1} \left(1-e^{-\xi Tt} \right)\right)^{\gamma-\delta}  \pi(d\xi)\label{eq:proof:SRD1Xnew11}\\
& \ + C_3 |\zeta| t A_T^{-1} T \int_0^\infty (\xi Tt)^{-1} \left(1-e^{-\xi Tt} \right) \1_{\{|\zeta A_T^{-1} \xi^{-1} \left(1-e^{-\xi Tt} \right)| > \varepsilon\}} \pi(d\xi).\label{eq:proof:SRD1Xnew1}
\end{align}

We consider now each term separately. For the first term we proceed as in the proof of \cite[Lemma 5.1]{GLT2018LimitInfVar}. If $\gamma\in (0,1)$, then from the inequality $x^{-1}(1-e^{-x})\leq 1$, $x>0$, we get
\begin{align*}
&C_1 \frac{1}{\gamma - \delta} |\zeta|^{\gamma-\delta} t^{\gamma-\delta} A_T^{-\gamma+\delta} T^{\gamma-\delta}\int_0^\infty \left((\xi Tt)^{-1} \left(1-e^{-\xi Tt} \right)\right)^{\gamma-\delta}  \pi(d\xi)\\
&\quad \leq  C_1 \frac{1}{\gamma - \delta} |\zeta|^{\gamma-\delta} t^{\gamma-\delta} T^{\gamma-\delta-1+\delta/\gamma} k^{\#}(T)^{(-\gamma+\delta)/\gamma}\\
&\quad \leq C_4  |\zeta|^{\gamma-\delta},
\end{align*}
since $T^{\gamma-\delta-1+\delta/\gamma} k^{\#}(T)^{(-\gamma+\delta)/\gamma} \to 0$ as $T\to \infty$, due to $\gamma-\delta - 1 + \delta/\gamma < 0$. If $\gamma \in (1,2)$, then from the inequality $x^{-1}(1-e^{-x})\leq x^{(1-\gamma)/(\gamma- \delta)}$ it follows
\begin{align*}
&C_1 \frac{1}{\gamma - \delta} |\zeta|^{\gamma-\delta} t^{\gamma-\delta} A_T^{-\gamma+\delta} T^{\gamma-\delta}\int_0^\infty \left((\xi Tt)^{-1} \left(1-e^{-\xi Tt} \right)\right)^{\gamma-\delta}  \pi(d\xi)\\
&\qquad \leq C_1 \frac{1}{\gamma - \delta} |\zeta|^{\gamma-\delta} t^{\gamma-\delta} T^{\gamma-\delta-1+\delta/\gamma} k^{\#}(T)^{(-\gamma+\delta)/\gamma} \int_0^\infty (\xi Tt)^{1-\gamma} \pi(d\xi)\\
&\qquad \leq C_1 \frac{1}{\gamma - \delta} |\zeta|^{\gamma-\delta} t^{1-\delta} T^{\delta/\gamma -\delta} k^{\#}(T)^{(-\gamma+\delta)/\gamma} \int_0^\infty \xi^{1-\gamma} \pi(d\xi)\\
&\qquad \leq C_5 |\zeta|^{\gamma-\delta},
\end{align*}
since $T^{\delta/\gamma -\delta} k^{\#}(T)^{(-\gamma+\delta)/\gamma}\to 0$ as $T\to \infty$ and $\int_0^\infty \xi^{1-\gamma} \pi(d\xi) < \infty$ due to \eqref{pifinitemean}. For $\gamma=1$ case we may use the fact that $x^{-1}(1-e^{-x})\leq x^{-\eta/(\gamma-\delta)}$, $\eta>0$, to obtain
\begin{align*}
&C_1 \frac{1}{\gamma - \delta} |\zeta|^{\gamma-\delta} t^{\gamma-\delta} A_T^{-\gamma+\delta} T^{\gamma-\delta}\int_0^\infty \left((\xi Tt)^{-1} \left(1-e^{-\xi Tt} \right)\right)^{\gamma-\delta}  \pi(d\xi)\\
&\qquad \leq C_1 \frac{1}{\gamma - \delta} |\zeta|^{\gamma-\delta} t^{1-\delta-\varepsilon} T^{-\varepsilon} k^{\#}(T)^{(-\gamma+\delta)/\gamma} \int_0^\infty \xi^{-\varepsilon} \pi(d\xi)\\
&\qquad \leq C_6 |\zeta|^{\gamma-\delta}.
\end{align*}

Returning now to the second term \eqref{eq:proof:SRD1Xnew1}, from the inequality $x^{-1}(1-e^{-x})\leq 1$, $x>0$, we get
\begin{align*}
&C_3 |\zeta| t A_T^{-1} T \int_0^\infty (\xi Tt)^{-1} \left(1-e^{-\xi Tt} \right) \1_{\{|\zeta A_T^{-1} \xi^{-1} \left(1-e^{-\xi Tt} \right)| > \varepsilon\}} \pi(d\xi)\\
&\qquad \leq C_3 |\zeta| t A_T^{-1} T \int_0^\infty \1_{\{|\zeta A_T^{-1} \xi^{-1} \left(1-e^{-\xi Tt} \right)| > \varepsilon\}} \pi(d\xi)\\
&\qquad \leq C_3 |\zeta| t A_T^{-1} T \int_0^\infty \1_{\{|\zeta| A_T^{-1} \xi^{-1} > \varepsilon\}}  \pi(d\xi)\\
&\qquad \leq C_3 |\zeta| t A_T^{-1} T \pi \left( \left(0, \varepsilon^{-1} A_T^{-1} |\zeta| \right) \right).
\end{align*}
By \eqref{regvarofp}, for arbitrary $0<\eta<1+\alpha-\gamma$, in some neighborhood of the origin it holds that $\pi \left( \left(0, x \right) \right) \leq C_7 x^{\alpha-\eta}$. Hence we have
\begin{align*}
&C_3 |\zeta| t A_T^{-1} T \int_0^\infty (\xi Tt)^{-1} \left(1-e^{-\xi Tt} \right) \1_{\{|\zeta A_T^{-1} \xi^{-1} \left(1-e^{-\xi Tt} \right)| > \varepsilon\}} \pi(d\xi)\\
&\qquad \leq C_8 |\zeta|^{1+\alpha-\eta} A_T^{-1-\alpha+\eta} T\\
&\qquad = C_8 |\zeta|^{1+\alpha-\eta} T^{1-(1+\alpha)/\gamma + \eta/\gamma}\\
&\qquad \leq C_9 |\zeta|^{1+\alpha-\eta}
\end{align*}
since $1+\alpha>\gamma$. We conclude finally from \eqref{eq:proof:SRD1Xnew11}-\eqref{eq:proof:SRD1Xnew1} that the following bound holds for $\left| \kappa_{\Delta X^*_{1,1}} (A_T^{-1} \zeta, Tt) \right|$
\begin{equation}\label{eq:thmSRD1:twopartsbound1}
\left| \kappa_{\Delta X^*_{1,1}} (A_T^{-1} \zeta, Tt) \right|  \leq C_5 |\zeta|^{\gamma-\delta} + C_9 |\zeta|^{1+\alpha-\eta}
\leq \begin{cases}
C_{10} |\zeta|^{\gamma-\delta}, & \ |\zeta|\leq 1,\\
C_{11} |\zeta|^{1+\alpha-\eta}, & \ |\zeta|> 1.\\
\end{cases}
\end{equation}

\item  We now consider $|\kappa_{\Delta X^*_{1,2}} (A_T^{-1} \zeta, Tt)|$ in \eqref{eq:thmSRD1:twoparts}. Because of \eqref{eq:proof:SRD1X1:ibragimov} we can write
\begin{equation*}
\kappa_{L_1} (\zeta) = \overline{k}(\zeta) \kappa_{\mathcal{S}_\gamma (\gamma^{1/\gamma} \sigma, \rho, 0)}(\zeta),
\end{equation*}
where $\overline{k}$ is slowly varying at zero such that $\overline{k}(\zeta)\sim k(1/\zeta)$ as $\zeta \to 0$ and $\kappa_{\mathcal{S}_\gamma (\gamma^{1/\gamma} \sigma, \rho, 0)}$ is a cumulant function of a stable distribution as in \eqref{cum:stable}. By \cite[Eq.~(34)]{GLT2018LimitInfVar} we have that
\begin{equation}\label{eq:thmDRSX1:proof:2}
\begin{aligned}
&\kappa_{\Delta X^*_{1,2}} (A_T^{-1} \zeta, Tt) = \kappa_{\mathcal{S}_\gamma (\gamma^{1/\gamma} \sigma, \rho, 0)}\left( \zeta \right)\\
&\  \times \int_0^\infty \int_{0}^{t}  \xi^{1-\gamma} \left( 1 - e^{-\xi T (t-s)} \right)^{\gamma} \frac{\overline{k}\left( \left(T k^{\#}(T) \right)^{-1/\gamma} \zeta \xi^{-1} \left( 1 - e^{-\xi T (t-s)} \right) \right)}{k^{\#}(T)}  ds \pi(d\xi).
\end{aligned}
\end{equation}
The definition of $k^{\#}$ implies that \cite[Theorem 1.5.13]{bingham1989regular}
\begin{equation*}
\frac{k^{\#} (T)}{ \overline{k} \left( \left( T k^{\#} (T) \right)^{-1/\gamma} \right)} \sim \frac{k^{\#} (T)}{ k \left( \left( T k^{\#} (T) \right)^{1/\gamma} \right)}  \to 1, \quad \text{ as } T\to \infty,
\end{equation*}
and due to slow variation of $\overline{k}$, for any $\zeta\in \R$, $\xi>0$ and $s\in (0,t)$, as $T\to \infty$
\begin{equation}\label{e:thmSRD1:proof1}
\begin{aligned}
&\frac{k^{\#} (T)}{\overline{k}\left( \left(T k^{\#}(T) \right)^{-1/\gamma} \zeta \xi^{-1} \left( 1 - e^{-\xi T (t-s)} \right) \right)} =\\
&\qquad \frac{\overline{k} \left( \left( T k^{\#} (T) \right)^{-1/\gamma} \right) }{\overline{k}\left( \left(T k^{\#}(T) \right)^{-1/\gamma} \zeta \xi^{-1} \left( 1 - e^{-\xi T (t-s)} \right) \right)}  \frac{k^{\#} (T)}{ \overline{k} \left( \left( T k^{\#} (T) \right)^{-1/\gamma} \right)} \to 1.
\end{aligned}
\end{equation}
By using Potter's bounds (see \cite[Theorem 1.5.6]{bingham1989regular}), we have from \eqref{e:thmSRD1:proof1} that for any $\varepsilon>0$
\begin{align*}
&\frac{\overline{k}\left( \left(T k^{\#}(T) \right)^{-1/\gamma} \zeta \xi^{-1} \left( 1 - e^{-\xi T (t-s)} \right) \right)}{k^{\#}(T)}\\
&\qquad \leq C_{12} \max \left\{ \zeta^\varepsilon \xi^{-\varepsilon} \left( 1 - e^{-\xi T (t-s)} \right)^\varepsilon , \zeta^{-\varepsilon} \xi^{\varepsilon} \left( 1 - e^{-\xi T (t-s)} \right)^{-\varepsilon}\right\}\\
&\qquad \leq C_{12} \left( 1 - e^{-\xi T (t-s)} \right)^{-\varepsilon} \max \left\{ \xi^{-\varepsilon}, \xi^\varepsilon \right\}  \max \left\{ \zeta^{-\varepsilon}, \zeta^\varepsilon \right\},
\end{align*}
for $T$ large enough. By taking $\varepsilon< \gamma$ we get
\begin{align*}
\xi^{1-\gamma} &\left( 1 - e^{-\xi T (t-s)} \right)^{\gamma} \frac{\overline{k}\left( \left(T k^{\#}(T) \right)^{-1/\gamma} \zeta \xi^{-1} \left( 1 - e^{-\xi T (t-s)} \right) \right)}{k^{\#}(T)}\\
&\leq  C_{12} \xi^{1-\gamma} \left( 1 - e^{-\xi T (t-s)} \right)^{\gamma-\varepsilon} \max \left\{ \xi^{-\varepsilon}, \xi^\varepsilon \right\} \max \left\{ \zeta^{-\varepsilon}, \zeta^\varepsilon \right\}\\
&\leq  C_{12} \xi^{1-\gamma} \max \left\{ \xi^{-\varepsilon}, \xi^\varepsilon \right\} \max \left\{ \zeta^{-\varepsilon}, \zeta^\varepsilon \right\}.
\end{align*}
Since $\gamma<1+\alpha$ and \eqref{pifinitemean} holds, we have
\begin{align*}
\int_0^\infty \int_{0}^{t}  \xi^{1-\gamma}  \max \left\{ \xi^{-\varepsilon}, \xi^\varepsilon \right\} ds \pi(d\xi) = t \int_0^1  \xi^{1-\gamma-\varepsilon} \pi(d\xi) + t \int_1^\infty \xi^{1-\gamma+\varepsilon} \pi(d\xi) < \infty.
\end{align*}
We finally conclude from \eqref{eq:thmDRSX1:proof:2} that
\begin{equation}\label{eq:thmSRD1:twopartsbound2}
\left|\kappa_{\Delta X^*_{1,2}} (A_T^{-1} \zeta, Tt) \right| \leq C_{13}  \left| \kappa_{\mathcal{S}_\gamma (\gamma^{1/\gamma} \sigma, \rho, 0)}\left( \zeta \right) \right| \max \left\{ \zeta^{-\varepsilon}, \zeta^\varepsilon \right\} \leq C_{14} |\zeta|^{\gamma} \max \left\{ \zeta^{-\varepsilon}, \zeta^\varepsilon \right\}.
\end{equation}

\item We shall now put the bounds for the terms in \eqref{eq:thmSRD1:twoparts} together. By using \eqref{eq:thmSRD1:twopartsbound1} and \eqref{eq:thmSRD1:twopartsbound2} one has from \eqref{eq:thmSRD1:twoparts} that
\begin{equation*}
| \kappa_{X_1^*} (A_T^{-1} \zeta, Tt) | \leq \begin{cases}
C_{10} |\zeta|^{\gamma-\delta} + C_{14} |\zeta|^{\gamma-\varepsilon}, & \ |\zeta|\leq 1,\\
C_{11} |\zeta|^{1+\alpha-\eta} + C_{14} |\zeta|^{\gamma+\varepsilon}, & \ |\zeta|> 1.\\
\end{cases}
\end{equation*}
Since $\gamma<1+\alpha$ and $\varepsilon$, $\delta$ and $\eta$ are arbitrary, we may choose them so that $\varepsilon<\delta<\gamma-q$ and $1+\alpha-\eta>\gamma+\varepsilon$, hence
\begin{equation}\label{e:kappabound11}
| \kappa_{X_1^*} (A_T^{-1} \zeta, Tt) | \leq \begin{cases}
C_{15} |\zeta|^{\gamma-\delta}, & \ |\zeta|\leq 1,\\
C_{16} |\zeta|^{1+\alpha-\eta}, & \ |\zeta|> 1.\\
\end{cases}
\end{equation}
This completes the derivation of the bound for $| \kappa_{X_1^*} (A_T^{-1} \zeta, Tt) |$

\item We now turn to \eqref{momSLPcase111} to get a bound for the moment $\E \left| A_T^{-1} \widetilde{X_1^*}(Tt)\right|^q$. We use \eqref{momSLPcase111}, \eqref{e:kappabound11} and
\begin{equation}\label{momSLPcase1112}
|\exp \kappa_{X_1^*} (A_T^{-1} \zeta, Tt)|^2 = \exp \{ 2 \Re \kappa_{X_1^*} (A_T^{-1} \zeta, Tt) \} \geq  \exp \{ - 2 |\kappa_{X_1^*} (A_T^{-1} \zeta, Tt) | \} ,
\end{equation}
and get
\begin{align*}
\E \left| A_T^{-1} \widetilde{X_1^*}(Tt)\right|^q &\leq k_q \int_{-\infty}^{\infty} \left( 1- \exp  \{ - 2 |\kappa_{X_1^*} (A_T^{-1} \zeta, Tt) | \} \right) |\zeta|^{-q-1} d \zeta\\
&\leq k_q \int_{|\zeta|\leq 1} \left( 1- \exp  \{ - 2 C_{15} \left|\zeta \right|^{\gamma - \delta} \} \right) |\zeta|^{-q-1} d \zeta\\
&\qquad \qquad + k_q \int_{|\zeta|>1} \left( 1- \exp  \{ - 2 C_{16} \left|\zeta \right|^{1+\alpha - \eta} \} \right) |\zeta|^{-q-1} d \zeta\\
&\leq k_q \int_{-\infty}^{\infty} \left( 1- \exp  \{ - 2 C_{15} \left|\zeta \right|^{\gamma - \delta} \} \right) |\zeta|^{-q-1} d \zeta\\
&\qquad \qquad + k_q \int_{-\infty}^{\infty} \left( 1- \exp  \{ - 2 C_{16} \left|\zeta \right|^{1+\alpha - \eta} \} \right) |\zeta|^{-q-1} d \zeta.
\end{align*}
By \eqref{e:thm:vonbahr2}, the terms on the right-hand side are $q$-th absolute moments of $(\gamma - \delta)$-stable and $(1+\alpha-\eta)$-stable random variables  with characteristic functions $\exp  \{ - 2 C_{15} \left|\zeta \right|^{\gamma - \delta} \}$ and $\exp  \{ - 2 C_{16} \left|\zeta \right|^{1+\alpha - \eta} \}$, respectively. Since $q<\gamma - \delta$ and $q<1+\alpha-\eta$, both integrals are finite. We conclude that the moment of the symmetrized integrated process is uniformly bounded. We now show this applies to the non-symmetrized process as well.

If $\gamma>1$, we may assume that $q>1$ and from \eqref{symmetricbound:geq1} we have
\begin{equation*}
\E \left| A_T^{-1} X_1^*(Tt)\right|^q \leq \E \left| A_T^{-1} \widetilde{X_1^*}(Tt)\right|^q.
\end{equation*}
If $\gamma\leq 1$, then from \eqref{symmetricbound:geq1}
\begin{equation*}
\E \left| A_T^{-1} X_1^*(Tt)\right|^q \leq \E \left| A_T^{-1} \widetilde{X_1^*}(Tt)\right|^q + 2 | \med (A_T^{-1} X_1^*(Tt))|^q.
\end{equation*}
Since $\{A_T^{-1} X_1^*(Tt) \}$ converges in distribution, the median $\med (A_T^{-1} X_1^*(Tt))$ also converges (see e.g.~\cite[Lemma 21.2]{van2000asymptotic}), hence we can bound the second term on the right. This completes the proof of uniform integrability of $\{|A_T^{-1} X_1^*(Tt)|^q \}$, hence the convergence of moments. Since the limiting process is $1/\gamma$-self-similar, from \cite[Theorem 1]{GLST2017Arxiv} we conclude that \begin{equation*}
\tau_{X_1^*}(q)=\frac{1}{\gamma}q, \ \text{ for } q<\gamma.
\end{equation*}
\end{itemize}
\end{proof}

For $\gamma>1+\alpha$ we have the following. 

\begin{lemma}\label{lemma:X1:2}
If Assumption \ref{assum} holds and $\gamma>1+\alpha$, then
\begin{equation}\label{e:tauX1case2}
\tau_{X_1^*}(q) \begin{cases}
=\frac{1}{1+\alpha} q, & 0<q\leq 1+\alpha,\\
\leq q-\alpha, & 1+\alpha < q < \gamma.
\end{cases}
\end{equation}
\end{lemma}

\begin{proof}
We first consider the case $q<1+\alpha$. The proof is similar to the proof of Lemma \ref{lemma:X1:1}. We will prove that $\{|A_T^{-1} X_1^*(Tt)|^q\}$ is uniformly integrable where now $A_T=T^{1/(1+\alpha)} \ell^{\#}\left(T \right)^{1/(1+\alpha)}$. We can assume $q>1$. From \eqref{symmetricbound:geq1}, \eqref{momSLPcase111} and \eqref{momSLPcase1112} it follows that
\begin{equation}\label{e:proofX1case2:1}
\begin{aligned}
\E \left| A_T^{-1} X_1^*(Tt)\right|^q &\leq \E \left| A_T^{-1} \widetilde{X_1^*}(Tt)\right|^q\\
&\leq k_q \int_{-\infty}^{\infty} \left( 1- \exp  \{ - 2 |\kappa_{X_1^*} (A_T^{-1} \zeta, Tt) | \} \right) |\zeta|^{-q-1} d \zeta.
\end{aligned}
\end{equation}
We now derive bound for $|\kappa_{X_1^*} (A_T^{-1} \zeta, Tt) |$. Again we use the decomposition \eqref{e:thmSRD1:decomposition} and bound $| \kappa_{\Delta X^*_{1,1}} (A_T^{-1} \zeta, Tt) |$ and $|\kappa_{\Delta X^*_{1,2}} (A_T^{-1} \zeta, Tt)|$ separately.

\begin{itemize}[leftmargin=*]
\item We consider first $\kappa_{\Delta X^*_{1,1}} (A_T^{-1} \zeta, Tt)$. From \eqref{eq:proof:SRD1X1:bound} we also have the following bound for $\varepsilon<1+\alpha-q$
\begin{equation*}
|\kappa_{L_1}(\zeta) | \leq C_1 |\zeta|^{1+\alpha-\varepsilon},
\end{equation*}
and by using Potter's bounds we have for $0<\delta<\varepsilon \alpha/(1+\alpha)$
\begin{equation*}
\widetilde{\ell}(T \xi^{-1}) = \frac{\widetilde{\ell}(T \xi^{-1})}{\widetilde{\ell}( \xi^{-1})} \widetilde{\ell}(\xi^{-1}) \leq C_2 \max \left\{ T^{-\delta}, T^{\delta} \right\} \widetilde{\ell}(\xi^{-1}).
\end{equation*}
By \eqref{regvarofp}, we can write the density $p$ of $\pi$ in the form $p(x)=\alpha \widetilde{\ell}(x^{-1}) x^{\alpha-1}$ with $\widetilde{\ell}$ slowly varying at infinity such that $\widetilde{\ell}(t) \sim \ell(t)$ as $t \to \infty$. Hence from \eqref{eq:proof:SRD1X1} we have
\begin{align*}
&\kappa_{\Delta X^*_{1,1}} (A_T^{-1} \zeta, Tt)  = \int_0^\infty \int_{-\infty}^{0} \kappa_{L_1} \left( \zeta A_T^{-1} T e^s \xi^{-1} \left(1-e^{-\xi t} \right)\right) ds \pi(T^{-1} d\xi)\\
&\quad = \int_0^\infty \int_{-\infty}^{0} \kappa_{L_1} \left( \zeta A_T^{-1} T e^s \xi^{-1} \left(1-e^{-\xi t} \right)\right)  \alpha \widetilde{\ell}(T \xi^{-1}) \xi^{\alpha-1} T^{-\alpha} ds d\xi.
\end{align*}
and
\begin{align}
\left| \kappa_{\Delta X^*_{1,1}} (A_T^{-1} \zeta, Tt) \right| & \leq C_3 |\zeta|^{1+\alpha-\varepsilon} T^{-(1+\alpha-\varepsilon)/(1+\alpha) + 1+\alpha-\varepsilon - \alpha + \delta}  \ell^{\#}\left(T \right)^{-1/(1+\alpha)} \nonumber\\
&\hspace{1cm} \times \int_0^\infty \int_{-\infty}^{0} e^{s} \left(\xi^{-1} \left(1-e^{-\xi t} \right) \right)^{1+\alpha-\varepsilon} \widetilde{\ell}(\xi^{-1}) \xi^{\alpha-1} ds d\xi \nonumber\\
&\ \leq C_3  |\zeta|^{1+\alpha-\varepsilon}  T^{-\varepsilon \alpha/(1+\alpha)+\delta}  \ell^{\#}\left(T \right)^{-1/(1+\alpha)} \int_0^\infty \widetilde{\ell}(\xi^{-1}) \xi^{\alpha-1} d\xi\nonumber\\
&\leq C_4 |\zeta|^{1+\alpha-\varepsilon}.\label{e:proofX1case2:4}
\end{align}

\item We consider now  $|\kappa_{\Delta X^*_{1,2}} (A_T^{-1} \zeta, Tt)|$. Analogous to \eqref{eq:proof:SRD1X1} we obtain
\begin{align*}
\kappa_{\Delta X^*_{1,2}} (A_T^{-1} \zeta, Tt) &= \int_0^\infty \int_{0}^{t} \kappa_{L_1} \left( \zeta A_T^{-1} \xi^{-1} \left( 1 - e^{-\xi T (t-s)} \right) \right) \xi T ds \pi(d\xi)\nonumber\\
&= \int_0^\infty \int_{0}^{t} \kappa_{L_1} \left( \zeta A_T^{-1} \xi^{-1} \left( 1 - e^{-\xi T (t-s)} \right) \right) \alpha \widetilde{\ell}(\xi^{-1}) \xi^{\alpha} T ds d\xi. 
\end{align*}
We shall assume that $\zeta>0$, the other case is similar. The change of variables $x=\zeta A_T^{-1} \xi^{-1}$ yields
\begin{equation}\label{e:proofX1case2:2}
\begin{aligned}
&\kappa_{\Delta X^*_{1,2}} (A_T^{-1} \zeta, Tt)\\
&= \zeta^{1+\alpha} \int_0^\infty \int_{0}^{t} \kappa_{L_1} \left( x \left( 1 - g_T(\zeta, x,s) \right) \right)  A_T^{-(1+\alpha)} T \widetilde{\ell}\left(A_T x \zeta^{-1}\right) \alpha x^{-\alpha-2} ds dx\\
&= \zeta^{1+\alpha} \int_0^\infty \int_{0}^{t} \kappa_{L_1} \left( x \left( 1 - g_T(\zeta, x,s) \right) \right) \frac{\widetilde{\ell}\left( T^{1/(1+\alpha)} \ell^{\#}\left(T \right)^{1/(1+\alpha)} x \zeta^{-1} \right) }{\ell^{\#} \left(T\right)} \alpha x^{-\alpha-2} ds dx,
\end{aligned}
\end{equation}
where $g_T(\zeta, x,s)=e^{-x^{-1} \frac{\zeta T}{A_T} (t-s)}$. From Potter's bounds, for $0<\eta<\min \left\{ \gamma-1-\alpha,\alpha \right\}$ there is $C_1$ such that
\begin{equation*}
\frac{\widetilde{\ell}\left( T^{1/(1+\alpha)} \ell^{\#}\left(T \right)^{1/(1+\alpha)} x \zeta^{-1} \right) }{\ell\left( T^{1/(1+\alpha)} \ell^{\#}\left(T \right)^{1/(1+\alpha)}\right)} \leq C_1 \max \left\{ x^{-\eta} \zeta^{\eta}, x^{\eta} \zeta^{-\eta} \right\}.
\end{equation*}
and by the definition of de Bruijn conjugate \cite[Theorem 1.5.13]{bingham1989regular}
\begin{equation*}
\frac{\ell^{\#} \left(T\right)}{\ell\left( \left( T\ell^{\#}\left(T \right) \right)^{1/(1+\alpha)} \right) } \sim 1, \text{  as } T\to \infty.
\end{equation*}
Hence, for $T$ large enough
\begin{equation*}
\frac{\widetilde{\ell}\left( T^{1/(1+\alpha)} \ell^{\#}\left(T \right)^{1/(1+\alpha)} x \zeta^{-1} \right) }{\ell^{\#} \left(T\right)} \leq C_2 \max \left\{ x^{-\eta} \zeta^{\eta}, x^{\eta} \zeta^{-\eta} \right\},
\end{equation*}
and by inserting this in \eqref{e:proofX1case2:2} we get
\begin{align*}
&\left| \kappa_{\Delta X^*_{1,2}} (A_T^{-1} \zeta, Tt)\right|\\
&\quad \leq \alpha C_2 \zeta^{1+\alpha} \max \left\{\zeta^{\eta}, \zeta^{-\eta} \right\}\\
&\hspace{1cm} \times \int_0^\infty \int_{0}^{t} \left| \kappa_{L_1} \left( x \left( 1 - g_T(\zeta, x,s) \right) \right) \right| \max \left\{ x^{-\eta}, x^{\eta}\right\}x^{-\alpha-2} ds dx.
\end{align*}
Now we use the bound \eqref{eq:proof:SRD1X1:bound} valid for arbitrary $\delta>0$ to obtain
\begin{equation}\label{e:proofX1case2:3}
\begin{aligned}
\left| \kappa_{\Delta X^*_{1,2}} (A_T^{-1} \zeta, Tt)\right| &\leq C_3 \zeta^{1+\alpha} \max \left\{\zeta^{\eta}, \zeta^{-\eta} \right\} \int_0^\infty \int_{0}^{t} \left( x \left( 1 - g_T(\zeta, x,s) \right) \right)^{\gamma-\delta}\\
&\quad \times \1_{\{x \left( 1 - g_T(\zeta, x,s)\right) \leq \varepsilon\}} \max \left\{ x^{-\eta}, x^{\eta}\right\}x^{-\alpha-2} ds dx\\
&+ C_4 \zeta^{1+\alpha} \max \left\{\zeta^{\eta}, \zeta^{-\eta} \right\} \int_0^\infty \int_{0}^{t} x \left( 1 - g_T(\zeta, x,s) \right)\\
&\quad \times \1_{\{x \left( 1 - g_T(\zeta, x,s)\right) > \varepsilon\}} \max \left\{ x^{-\eta}, x^{\eta}\right\}x^{-\alpha-2} ds dx\\
&=: I_1 + I_2.
\end{aligned}
\end{equation}
We consider each term separately.

\begin{itemize}[label=$\circ$]
\item For $I_{1}$ we make change of variables $y=x \left( 1 - g_T(\zeta, x,s) \right)$ and get
\begin{align*}
I_{1} &= C_3 \zeta^{1+\alpha} \max \left\{\zeta^{\eta}, \zeta^{-\eta} \right\} \int_0^\infty \int_{0}^{t} y^{\gamma-\delta} \1_{\{y\leq \varepsilon\}}\\
&\quad \times \max \left\{ y^{-\eta}  \left( 1 - g_T(\zeta, x,s) \right)^{\eta}, y^{\eta} \left( 1 - g_T(\zeta, x,s) \right)^{-\eta}\right\} y^{-\alpha-2}\\
&\qquad \qquad \times \left( 1 - g_T(\zeta, x,s) \right)^{\alpha+1} ds dy\\
&\leq C_3 \zeta^{1+\alpha} \max \left\{\zeta^{\eta}, \zeta^{-\eta} \right\} \int_0^\infty \int_{0}^{t} y^{\gamma-\alpha-2-\delta-\eta} \1_{\{y\leq \varepsilon\}}(y) \\
&\qquad \qquad \times \left( 1 - g_T(\zeta, x,s) \right)^{\alpha+1-\eta} ds dy\\
&\leq C_5 \zeta^{1+\alpha} \max \left\{\zeta^{\eta}, \zeta^{-\eta} \right\} \int_0^\varepsilon y^{\gamma-\alpha-2-\delta-\eta} dy\\
&= C_6 \zeta^{1+\alpha} \max \left\{\zeta^{\eta}, \zeta^{-\eta} \right\},
\end{align*}
where we have used the fact that the integral in the last line is finite due to $\gamma>1+\alpha$ and the choice of $\eta$ and $\delta$.

\item Consider now $I_2$. Since $x \left( 1 - g_T(\zeta, x,s)\right) > \varepsilon$ implies $x>\varepsilon$, we have for $I_2$,
\begin{align*}
I_{2} &\leq C_7 \zeta^{1+\alpha} \max \left\{\zeta^{\eta}, \zeta^{-\eta} \right\} \int_0^\infty \int_{0}^{t} x^{-\alpha-1} \1_{\{x > \varepsilon\}} \max \left\{ x^{-\eta}, x^{\eta}\right\} ds dx\\
&\leq C_8 \zeta^{1+\alpha} \max \left\{\zeta^{\eta}, \zeta^{-\eta} \right\} \int_{\varepsilon}^\infty x^{-\alpha-1+\eta} dx\\
&= C_9 \zeta^{1+\alpha} \max \left\{\zeta^{\eta}, \zeta^{-\eta} \right\}.
\end{align*}
\end{itemize}

Returning back to \eqref{e:proofX1case2:3} we conclude that
\begin{equation}\label{e:proofX1case2:5}
\left| \kappa_{\Delta X^*_{1,2}} (A_T^{-1} \zeta, Tt)\right| \leq C_{10} \zeta^{1+\alpha} \max \left\{\zeta^{\eta}, \zeta^{-\eta} \right\}.
\end{equation}

\medskip

From \eqref{eq:thmSRD1:twoparts}, \eqref{e:proofX1case2:4} and \eqref{e:proofX1case2:5} we get the bound for $|\kappa_{X_1^*} (A_T^{-1} \zeta, Tt) |$. Namely, for $\varepsilon>0$ and $\eta>0$ arbitrary small there are constants $C_1, C_2 >0$ such that
\begin{equation*}
|\kappa_{X_1^*} (A_T^{-1} \zeta, Tt) | \leq \begin{cases}
C_1 |\zeta|^{1+\alpha-\varepsilon} + C_2 |\zeta|^{1+\alpha-\eta}, & |\zeta|\leq 1,\\
C_1 |\zeta|^{1+\alpha-\varepsilon} + C_2 |\zeta|^{1+\alpha+\eta}, & |\zeta|> 1.
\end{cases}
\end{equation*}
Assuming e.g.~that $\varepsilon<\eta$ we have
\begin{equation*}
|\kappa_{X_1^*} (A_T^{-1} \zeta, Tt) | \leq \begin{cases}
C_3 |\zeta|^{1+\alpha-\eta}, & |\zeta|\leq 1,\\
C_4 |\zeta|^{1+\alpha+\eta}, & |\zeta|> 1.
\end{cases}
\end{equation*}
We use this to get the bound for the $q$-th absolute moment as in the proof of Lemma \ref{lemma:X1:1}. It follows from \eqref{e:proofX1case2:1} that
\begin{align*}
\E \left| A_T^{-1} X_1^*(Tt)\right|^q &\leq k_q \int_{|\zeta|\leq 1} \left( 1- \exp  \{ - 2 C_3 |\zeta|^{1+\alpha-\eta} \} \right) |\zeta|^{-q-1} d \zeta\\
&\qquad \qquad + k_q \int_{|\zeta|> 1} \left( 1- \exp  \{ - 2 C_4 |\zeta|^{1+\alpha+\eta} \} \right) |\zeta|^{-q-1} d \zeta\\
&\leq k_q \int_{-\infty}^\infty \left( 1- \exp  \{ - 2 C_3 |\zeta|^{1+\alpha-\eta} \} \right) |\zeta|^{-q-1} d \zeta\\
&\qquad \qquad + k_q \int_{-\infty}^\infty \left( 1- \exp  \{ - 2 C_4 |\zeta|^{1+\alpha+\eta} \} \right) |\zeta|^{-q-1} d \zeta.
\end{align*}
The terms on the right-hand side are $q$-th absolute moments of $(1+\alpha-\eta)$-stable and $(1+\alpha+\eta)$-stable random variables  with characteristic functions $\exp  \{ - 2 C_3 |\zeta|^{1+\alpha-\eta} \}$ and $\exp  \{ - 2 C_4 |\zeta|^{1+\alpha+\eta} \}$, respectively. We are considering the case $q<1+\alpha$, hence these moments are finite if we choose $\eta$ small enough. Hence, $\{|A_T^{-1} X_1^*(Tt)|^q \}$ is uniformly integrable, the moments converge and from \cite[Theorem 1]{GLST2017Arxiv} we have that $\tau_{X_1^*}(q)=q/(1+\alpha)$ for $q<1+\alpha$. Since the scaling function is convex (see e.g.~\cite{GLST2016JSP}), hence continuous, we obtain
\begin{equation*}
\tau_{X_1^*}(q)=\frac{1}{1+\alpha} q, \quad  \text{ for } q\leq 1+\alpha.
\end{equation*}

\item We now turn to the case $1+\alpha<q<\gamma$ in Lemma \ref{lemma:X1:2}. We will show that for arbitrary $\varepsilon>0$
\begin{equation}\label{e:proof:X1:case2:toprove}
\E \left| T^{-1+\frac{\alpha}{q}-\frac{\varepsilon}{q}} X_1^*(T)\right|^q \leq C,
\end{equation}
for some constant $C>0$ and $T$ large enough. This implies that $\tau_{X_1^*}(q)\leq q-\alpha+\varepsilon$ and completes the proof since $\varepsilon$ is arbitrary. To show \eqref{e:proof:X1:case2:toprove}, we will use \eqref{e:proofX1case2:1} with $A_T=T^{1-\alpha/q+\varepsilon/q}$. First, by \eqref{e:thmSRD1:decomposition} and \eqref{integrationrule}, we may express the cumulant function of $X^*_1(T)$ as
\begin{align*}
\kappa_{A_T^{-1}X_1^*(T)}(\zeta) &= \int_0^\infty \int_{-\infty}^0 \kappa_{L_1} \left( A_T^{-1} \zeta \xi^{-1} e^{\xi s} \left(1 - e^{-\xi T} \right)  \right) ds \xi \pi(d\xi)\\
&\qquad  + \int_0^\infty \int_{0}^T \kappa_L \left( A_T^{-1} \zeta \xi^{-1} \left(1 - e^{-\xi (T - s)} \right)  \right) ds \xi \pi(d\xi).
\end{align*}
Making a change of variables and writing $p(x)=\alpha \widetilde{\ell}(x^{-1}) x^{\alpha-1}$, with $\widetilde{\ell}(t) \sim \ell(t)$ as $t \to \infty$, yields
\begin{align*}
\kappa_{A_T^{-1}X_1^*(T)}(\zeta) &= \int_0^\infty \int_{-\infty}^0 \kappa_{L_1} \left( T^{\frac{\alpha}{q}-\frac{\varepsilon}{q}} \zeta x^{-1} e^{x \frac{s}{T}} \left(1 - e^{-x} \right)  \right) ds x T^{-1} \pi(T^{-1} dx)\\
&\qquad  + \int_0^\infty \int_{0}^T \kappa_L \left( T^{\frac{\alpha}{q}-\frac{\varepsilon}{q}}  \zeta x^{-1} \left(1 - e^{-x\left(1- \frac{s}{T}\right)} \right)  \right) ds x T^{-1} \pi(T^{-1} dx)\\
&= \int_0^\infty \int_{-\infty}^0 \kappa_{L_1} \left( T^{\frac{\alpha}{q}-\frac{\varepsilon}{q}} \zeta x^{-1} e^{x u} \left(1 - e^{-x} \right)  \right) du x \pi(T^{-1} dx)\\
&\qquad  + \int_0^\infty \int_{0}^1 \kappa_L \left( T^{\frac{\alpha}{q}-\frac{\varepsilon}{q}}  \zeta x^{-1} \left(1 - e^{-x u} \right)  \right) du x \pi(T^{-1} dx)\\
&= \int_0^\infty \int_{-\infty}^0 \kappa_{L_1} \left( T^{\frac{\alpha}{q}-\frac{\varepsilon}{q}} \zeta x^{-1} e^{x u} \left(1 - e^{-x} \right)  \right) du \alpha \widetilde{\ell}(Tx^{-1}) x^{\alpha} T^{-\alpha} dx\\
&\qquad  + \int_0^\infty \int_{0}^1 \kappa_L \left( T^{\frac{\alpha}{q}-\frac{\varepsilon}{q}}  \zeta x^{-1} \left(1 - e^{-x u} \right)  \right) du \alpha \widetilde{\ell}(Tx^{-1}) x^{\alpha} T^{-\alpha} dx.
\end{align*}
Take $\delta>0$ such that $q+\delta<\gamma$ and $\delta<\frac{\varepsilon q}{\alpha-\varepsilon}$ and note that from \eqref{eq:proof:SRD1X1:bound} we have the bound
\begin{equation*}
|\kappa_{L_1}(\zeta) | \leq C |\zeta|^{q+\delta}, \quad \zeta \in \R.
\end{equation*}
Hence,
\begin{align*}
&\left| \kappa_{A_T^{-1}X_1^*(T)}(\zeta) \right|\\
&\ \leq C |\zeta|^{q+\delta}  \int_0^\infty \int_{-\infty}^0 x^{\alpha-q-\delta} e^{(q+\delta)x u} \left(1 - e^{-x} \right)^{q+\delta} \alpha \widetilde{\ell}(Tx^{-1}) T^{\left(\frac{\alpha}{q}-\frac{\varepsilon}{q}\right)\left(q+\delta\right)-\alpha} du dx\\
&\ + C |\zeta|^{q+\delta} \int_0^\infty \int_{0}^1  x^{\alpha-q-\delta} \left(1 - e^{-x u} \right)^{q+\delta} \alpha \widetilde{\ell}(Tx^{-1}) T^{\left(\frac{\alpha}{q}-\frac{\varepsilon}{q}\right)\left(q+\delta\right)-\alpha} du dx.
\end{align*}
Note that by the choice of $\delta$, we have $\left(\frac{\alpha}{q}-\frac{\varepsilon}{q}\right)\left(q+\delta\right)-\alpha<0$. By Potter's bounds, for any $\eta>0$ we have that $\widetilde{\ell}(Tx^{-1}) \leq C_1 \widetilde{\ell}(x^{-1}) T^{\eta}$. Taking $\eta< \alpha- \left(\frac{\alpha}{q}-\frac{\varepsilon}{q}\right)\left(q+\delta\right)$ yields
\begin{align*}
&\left| \kappa_{A_T^{-1}X_1^*(T)}(\zeta) \right|\\
&\leq C_2 T^{\left(\frac{\alpha}{q}-\frac{\varepsilon}{q}\right)\left(q+\delta\right)-\alpha+\eta}  |\zeta|^{q+\delta}  \int_0^\infty x^{\alpha-q-1-\delta} \left(1 - e^{-x} \right)^{q+\delta} \alpha \widetilde{\ell}(x^{-1}) dx\\
&\qquad  + C_3 T^{\left(\frac{\alpha}{q}-\frac{\varepsilon}{q}\right)\left(q+\delta\right)-\alpha+\eta}  |\zeta|^{q+\delta}  \int_0^\infty \int_{0}^1  x^{\alpha-q-\delta} \left(1 - e^{-x u} \right)^{q+\delta}  \alpha \widetilde{\ell}(x^{-1}) du dx\\
&\leq C_2 |\zeta|^{q+\delta}  \int_0^\infty x^{\alpha-1} \alpha \widetilde{\ell}(x^{-1}) dx\\
&\qquad  + C_3 |\zeta|^{q+\delta}  \int_0^\infty \int_{0}^1  x^{\alpha} u^{q-\varepsilon} \alpha \widetilde{\ell}(x^{-1}) du dx\\
&\leq C_2 |\zeta|^{q+\delta} + C_4 |\zeta|^{q+\delta}  \int_0^\infty x\pi(dx)\\
&\leq C_5 |\zeta|^{q+\delta},
\end{align*}
where we have used the inequality $x^{-1}(1-e^{-x})\leq 1$, $x>0$, \eqref{pifinitemean} and the fact that $\pi$ is probability measure. This completes the derivation of the bound for $\left| \kappa_{A_T^{-1}X_1^*(T)}(\zeta) \right|$. Now we use \eqref{e:proofX1case2:1} to get that
\begin{equation*}
\E \left| T^{-1+\frac{\alpha}{q}-\frac{\varepsilon}{q}} X_1^*(T)\right|^q \leq k_q \int_{-\infty}^{\infty} \left( 1- \exp  \{ - 2 C_5 |\zeta|^{q+\delta} \} \right) |\zeta|^{-q-1} d \zeta.
\end{equation*}
The right hand side corresponds to the $q$-th absolute moment of $(q+\delta)$-stable random variable which is finite. Hence, \eqref{e:proof:X1:case2:toprove} holds and this completes the proof.
\end{itemize}
\end{proof}

In case $\gamma>1+\alpha$, for the moments of order $q$ in the range $(1+\alpha,\gamma)$ we are not able to obtain the exact form of the scaling function $\tau_{X_1^*}(q)$ in Lemma \ref{lemma:X1:2}. However, we provide a bound which will be enough for the proof of the main results later on. We conjecture that equality holds in \eqref{e:tauX1case2}. The proofs of Lemma \ref{lemma:X1:1} and Lemma \ref{lemma:X1:2} are particularly delicate because of the presence of infinite second moments. 

\subsubsection{The scaling function of $X_2^*$}

By the decomposition \eqref{e:decomposition}, $X_2^*$ is the integrated supOU process corresponding to a characteristic quadruple $(0,0,\mu_2,\pi)$ where $\mu_{2}(dx)=\mu(dx) \1_{\{|x|\leq1\}}(dx)$ and we assume $\mu_2\not\equiv 0$. In particular, $X_2^*$ has finite variance since $\int_{|x|>1} x^2 \mu_2(dx)<\infty$. Moreover, $\int_{|x|>1} e^{a|x|} \mu_2(dx)<\infty$ and exponential moment of $X_2(1)$ is finite which by \cite[Theorem 7.2.1]{lukacs1970characteristic} implies that the cumulant function of $X_2(1)$ is analytic in the neighborhood of the origin and all moments are finite. Hence, we may use the results of \cite{GLT2017Limit}, namely Eq.~(4.9), Theorem 4.2 and Theorem 4.3 from \cite{GLT2017Limit}. These results are stated here in the following lemma.

\begin{lemma}\label{lemma:X2}
Suppose that Assumption \ref{assum} holds. Then the scaling function $\tau_{X_2^*}(q)$ of the process $X_2^*$ is as follows:
\begin{enumerate}[(a)]
	\item If $\alpha>1$, then
	\begin{equation*}
	\tau_{X_2^*}(q)=\begin{cases}
	\frac{1}{2} q, & 0<q\leq q_*,\\
	q-\alpha, & q\geq q^*.
	\end{cases}
	\end{equation*}
	where $q_*$ is the largest even integer less than or equal to $2\alpha$ and $q^*$ is the smallest even integer greater than $2\alpha$.
	\item If $\alpha\in(0,1)$ and $\beta<1+\alpha$, then
	\begin{equation*}
	\tau_{X_2^*}(q) = \begin{cases}
	\frac{1}{1+\alpha} q, & 0<q\leq 1+\alpha,\\
	q-\alpha, & q \geq 1+\alpha.
	\end{cases}
	\end{equation*}
	\item If $\alpha\in(0,1)$ and $1+\alpha<\beta<2$, then
	\begin{equation*}
	\tau_{X_2^*}(q) = \begin{cases}
	\left(1-\frac{\alpha}{\beta}\right) q, & 0<q\leq \beta,\\
	q-\alpha, & q \geq \beta.
	\end{cases}
	\end{equation*}
\end{enumerate}
\end{lemma}

Lemma \ref{lemma:X2}(a) and convexity of the scaling function imply that for $q_*\leq q \leq q^*$
\begin{equation*}
\tau_{X_2^*}(q) \leq \frac{q^*-\alpha-q_*/2}{q^*-q_*} (x-q_*) + \frac{q_*}{2}.
\end{equation*}
Note also that Lemma \ref{lemma:X2}(a) implies that $\tau_{X_2^*}(q)=q/2$ for $q\leq 2$ which will be enough for the proofs of Theorems \ref{thm:mainb=0} and \ref{thm:mainb!=0} below.

In contrast with the component $X_1^*$, the scaling function of $X_2^*$ displays intermittency in any case covered by Lemma \ref{lemma:X2}. Even in the short-range dependent scenario $\alpha>1$, intermittency appears for higher order moments. 

\subsubsection{The scaling function of $X_3^*$}

The process $X_3^*$ defined in \eqref{e:decomposition} is a Gaussian process. Its scaling function is given in \cite[Theorem 4.1 and 4.4]{GLT2017Limit}. Gaussian supOU processes do not display intermittency and their scaling function is linear over positive reals.
This result is stated here in Lemma \ref{lemma:X3}.

\begin{lemma}\label{lemma:X3}
Suppose that Assumption \ref{assum} holds. Then the scaling function $\tau_{X_3^*}(q)$ of the process $X_3^*$ is as follows:
\begin{enumerate}[(a)]
	\item If $\alpha>1$, then
	\begin{equation*}
	\tau_{X_3^*}(q) = \frac{1}{2} q, \quad \forall q>0.
	\end{equation*}
	\item If $\alpha\in(0,1)$, then
	\begin{equation*}
	\tau_{X_3^*}(q) = \left(1-\frac{\alpha}{2}\right) q, \quad \forall q>0.
	\end{equation*}
\end{enumerate}
\end{lemma}


\subsection{The scaling function of the integrated process $X^*$}\label{s:sfX*}

To derive the scaling function of the integrated process $X^*=X_1^*+X_2^*+X_3^*$ we will use the expressions for the scaling functions of components in the decomposition \eqref{e:decomposition} and the following proposition which shows how to compute the scaling function of a sum of independent processes.

\begin{proposition}\label{prop:sfofsum}
Let $Y_1=\{Y_1(t),\, t \geq 0\}$ and $Y_2=\{Y_2(t),\, t \geq 0\}$ be two independent processes with the scaling functions $\tau_{Y_1}$ and $\tau_{Y_2}$, respectively, and suppose that $\E Y_1(t) = \E Y_2(t)=0$ for every $t\geq 0$ if the mean is finite. Suppose $q\in (0,\overline{q}(Y_1)) \cap (0,\overline{q}(Y_2))$ and $\tau_{Y_1}(q)$ and $\tau_{Y_2}(q)$ are well-defined and positive. If $q<1$, assume additionally that  $\tau_{Y_1}(q) \neq \tau_{Y_2}(q)$. Then the scaling function of the sum $Y_1+Y_2=\{Y_1(t)+Y_2(t),\, t \geq 0\}$, evaluated at point $q$, equals
\begin{equation*}
\tau_{Y_1+Y_2}(q) = \max \left\{ \tau_{Y_1}(q), \tau_{Y_2}(q) \right\}.
\end{equation*}
\end{proposition}

\begin{proof}
Suppose that $\max \left\{ \tau_{Y_1}(q), \tau_{Y_2}(q) \right\}= \tau_{Y_1}(q)$. For $\varepsilon>0$ we can take $t$ large enough so that
\begin{equation*}
\frac{\log \E \left|Y_1(t)\right|^q }{\log t} \geq \frac{\log \E \left|Y_2(t)\right|^q }{\log t} - \varepsilon
\end{equation*}
and hence
\begin{equation}\label{e:1>2}
\E \left|Y_1(t)\right|^q \geq \E \left|Y_2(t)\right|^q  t^{-\varepsilon}.
\end{equation}
From the inequality
\begin{equation}\label{cq-inequality}
\E \left|Y_1(t)+Y_2(t) \right|^q \leq c_q \E \left|Y_1(t)\right|^q + c_q \E \left|Y_2(t)\right|^q, \quad c_q=\max \left\{1, 2^{q-1} \right\},
\end{equation}
we have that
\begin{align*}
\tau_{Y_1+Y_2}(q)  &= \lim_{t\to \infty} \frac{\log \E \left|Y_1(t)+Y_2(t) \right|^q}{\log t}\\
&\leq \lim_{t\to \infty} \left( \frac{\log c_q}{\log t} + \frac{\log \left( \E \left|Y_1(t)\right|^q + \E \left|Y_2(t)\right|^q \right)}{\log t} \right)\\
&= \lim_{t\to \infty} \frac{\log \E \left|Y_1(t)\right|^q  + \log \left(1 + \frac{\E \left|Y_2(t)\right|^q}{\E \left|Y_1(t)\right|^q} \right) }{\log t}\\
&\leq \lim_{t\to \infty} \frac{\log \E \left|Y_1(t)\right|^q  + \log \left(1 + t^{\varepsilon}\right) }{\log t}\\
&=\tau_{Y_1}(q) + \varepsilon,
\end{align*}
where we used \eqref{e:1>2}. Since $\varepsilon$ was arbitrary, we conclude that $\tau_{Y_1+Y_2}(q)\leq \max \left\{ \tau_{Y_1}(q), \tau_{Y_2}(q) \right\}$.

We prove the reverse inequality for the $q\geq 1$ case first. Note that in this case $\E Y_1(t) = \E Y_2(t)=0$ for every $t\geq 0$. For $x\in \R$ we have by using Jensen's inequality that
\begin{equation*}
|x|^q = \left| x +  \E Y_2(t) \right|^q \leq \E \left| x +  Y_2(t) \right|^q.
\end{equation*}
Letting $F_{Y_1(t)}$ and $F_{Y_2(t)}$ denote the distribution functions of $Y_1(t)$ and $Y_2(t)$, respectively, we get by independence
\begin{align*}
\E \left|Y_1(t)\right|^q &= \int_{-\infty}^\infty |x|^q dF_{Y_1(t)} (x) \leq \int_{-\infty}^\infty \E \left| x +  Y_2(t) \right|^q dF_{Y_1(t)} (x)\\
& = \int_{-\infty}^\infty \int_{-\infty}^\infty \left| x +  y \right|^q dF_{Y_2(t)} (y) dF_{Y_1(t)} (x) = \E \left|Y_1(t)+Y_2(t) \right|^q.
\end{align*}
From here it follows that
\begin{equation*}
\tau_{Y_1+Y_2}(q) \geq \tau_{Y_1}(q).
\end{equation*}

Suppose now that $q<1$ and let $Y_2'=\{Y_2'(t),\, t \geq 0\}$ be an independent copy of the process $Y_2=\{Y_2(t),\, t \geq 0\}$, independent of $Y_1$. From \eqref{cq-inequality} we have that
\begin{equation}\label{lemma:proof:e1}
\E \left|Y_1(t)+Y_2(t) \right|^q \geq \E \left|Y_1(t)+Y_2(t) - Y_2'(t) \right|^q - \E \left|Y_2(t)\right|^q.
\end{equation}
Since $Y_2(t) - Y_2'(t)$ is symmetric it follows that $Y_1(t)+Y_2(t) - Y_2'(t) \overset{d}{=} Y_1(t) - Y_2(t) + Y_2'(t)$. From the identity
\begin{equation*}
Y_1(t) = \frac{1}{2} \left( Y_1(t) + Y_2(t) - Y_2'(t) + Y_1(t) - Y_2(t) + Y_2'(t) \right)
\end{equation*}
we get by using \eqref{cq-inequality} that
\begin{align*}
\E \left|Y_1(t)\right|^q &\leq 2^{-q} \left( \E \left|Y_1(t)+Y_2(t) - Y_2'(t) \right|^q + \E \left|Y_1(t) - Y_2(t) + Y_2'(t) \right|^q \right)\\
&= 2^{1-q} \E \left|Y_1(t)+Y_2(t) - Y_2'(t) \right|^q.
\end{align*}
Returning back to \eqref{lemma:proof:e1} we have
\begin{equation}\label{lemma:proof:e2}
\E \left|Y_1(t)+Y_2(t) \right|^q \geq 2^{q-1} \E \left|Y_1(t)\right|^q - \E \left|Y_2(t)\right|^q = \E \left|Y_1(t)\right|^q \left( 2^{q-1} - \frac{\E \left|Y_2(t)\right|^q}{\E \left|Y_1(t)\right|^q} \right).
\end{equation}
We assumed that $\tau_{Y_1}(q)\neq \tau_{Y_2}(q)$ and without loss of generality let $\tau_{Y_1}(q)> \tau_{Y_2}(q)$. For $\varepsilon>0$ small enough we can take $t$ large enough so that
\begin{equation*}
\frac{\log \E \left|Y_1(t)\right|^q }{\log t} \geq \frac{\log \E \left|Y_2(t)\right|^q }{\log t} + \varepsilon
\end{equation*}
and hence
\begin{equation*}
\E \left|Y_1(t)\right|^q \geq \E \left|Y_2(t)\right|^q  t^{\varepsilon}.
\end{equation*}
We conclude that
\begin{equation*}
\frac{\E \left|Y_2(t)\right|^q}{\E \left|Y_1(t)\right|^q} \to 0, \text{ as } t\to \infty.
\end{equation*}
By taking logarithms in \eqref{lemma:proof:e2}, dividing by $\log t$ and letting $t\to \infty$, we get
\begin{equation*}
\tau_{Y_1+Y_2}(q) \geq \tau_{Y_1}(q).
\end{equation*}
\end{proof}

We are now ready for the proofs of the main results.

\begin{proof}[Proof of Theorem \ref{thm:mainb=0}]
We shall combine the results of Lemmas \ref{lemma:X1:1}, \ref{lemma:X1:2} and \ref{lemma:X2} by using Proposition \ref{prop:sfofsum}.\\

\textit{(a)} Suppose that $\gamma<1+\alpha$ and split cases depending on the scaling function of $X_2^*$.
\begin{itemize}
\item If $\alpha>1$, then from Lemma \ref{lemma:X2} $\tau_{X_2^*}(q)=q/2$ for $q\in(0,2)$. Since $1/\gamma>1/2$, we have for $q\in (0,\gamma)$
\begin{equation*}
\tau_{X^*}(q) = \max \left\{ \tau_{X_1^*}(q), \tau_{X_2^*}(q) \right\} = \max \left\{ \frac{1}{\gamma} q, \frac{1}{2}q \right\} = \frac{1}{\gamma} q.
\end{equation*}
\item If $\alpha\in(0,1)$ and $\beta<1+\alpha$, then we have for $q\in (0,\gamma)$
\begin{equation*}
\tau_{X^*}(q) = \max \left\{ \tau_{X_1^*}(q), \tau_{X_2^*}(q) \right\} = \max \left\{ \frac{1}{\gamma} q, \frac{1}{1+\alpha}q \right\} = \frac{1}{\gamma} q,
\end{equation*}
since $\frac{1}{\gamma} > \frac{1}{1+\alpha}$.
\item If $\alpha\in(0,1)$ and $\beta>1+\alpha$, then for $q\in (0,\gamma)$
\begin{equation*}
\tau_{X^*}(q) = \max \left\{ \tau_{X_1^*}(q), \tau_{X_2^*}(q) \right\} = \max \left\{ \frac{1}{\gamma} q, \left(1-\frac{\alpha}{\beta}\right) q \right\} = \frac{1}{\gamma} q,
\end{equation*}
since $1-\frac{\alpha}{\beta}<1+\frac{1-\gamma}{\beta} < 1+\frac{1-\gamma}{\gamma} = \frac{1}{\gamma}$.
\end{itemize}

\textit{(b)} If $\gamma>1+\alpha$ and $\beta<1+\alpha$, then necessarily $\alpha \in (0,1)$. For $1\leq q \leq 1+\alpha$ we have by Proposition \ref{prop:sfofsum} and by Lemmas \ref{lemma:X1:2} and \ref{lemma:X2} that
\begin{equation}\label{e:stepinproof1}
\tau_{X^*}(q) = \max \left\{ \tau_{X_1^*}(q), \tau_{X_2^*}(q) \right\} = \max \left\{ \frac{1}{1+\alpha} q, \frac{1}{1+\alpha} q \right\} = \frac{1}{1+\alpha} q.
\end{equation}
Since $\tau_{X_1^*}(q)=\tau_{X_2^*}(q)$ we cannot use Proposition \ref{prop:sfofsum} for $q<1$, but from \eqref{e:stepinproof1}, $\tau_{X^*}(0)=0$ and the fact that the scaling function is always convex, we conclude using \cite[Lemma 2]{GLST2017Arxiv} that $\tau_{X^*}(q) =  \frac{1}{1+\alpha} q$ for $q<1$ also.

For $1+\alpha<q<\gamma$ we have
\begin{equation*}
\tau_{X^*}(q) = \max \left\{ \tau_{X_1^*}(q), \tau_{X_2^*}(q) \right\} = \max \left\{ \frac{1}{1+\alpha} q, q-\alpha \right\} = q-\alpha.
\end{equation*}

\textit{(c)} If $\gamma>1+\alpha$, $\beta>1+\alpha$ and $\beta\leq \gamma$,  we have
\begin{equation*}
\tau_{X^*}(q) = \begin{cases}
\max \left\{ \frac{1}{1+\alpha} q, \left(1-\frac{\alpha}{\beta} \right) q \right\}, & 0<q\leq 1+\alpha,\\
\max \left\{ \tau_{X_1^*}(q), \left(1-\frac{\alpha}{\beta} \right) q \right\}, & 1+\alpha < q \leq \beta,\\
\max \left\{ \tau_{X_1^*}(q), q-\alpha \right\}, & \beta < q < \gamma.
\end{cases}
\end{equation*}
For the case $q\leq 1+\alpha$, note that because $\beta>1+\alpha$ we have $1-\frac{\alpha}{\beta}>1-\frac{\alpha}{1+\alpha}=\frac{1}{1+\alpha}$. In Lemma \ref{lemma:X1:2} we showed that $\tau_{X_1^*}(q)\leq q-\alpha$ for $1+\alpha<q<\gamma$ and for $q\leq \beta$ we have $q- \frac{\alpha}{\beta}q \geq q-\alpha$. Hence we obtain
\begin{equation*}
\tau_{X^*}(q) = \begin{cases}
\left(1-\frac{\alpha}{\beta} \right) q, & 0<q\leq 1+\alpha,\\
\left(1-\frac{\alpha}{\beta} \right) q , & 1+\alpha < q \leq \beta,\\
q-\alpha, & \beta < q < \gamma.
\end{cases}
\end{equation*}

\textit{(d)} If $\gamma>1+\alpha$, $\beta>1+\alpha$ and $\beta> \gamma$, then by using the same arguments as in the previous case we get
\begin{equation*}
\tau_{X^*}(q) = \begin{cases}
\max \left\{ \frac{1}{1+\alpha} q, \left(1-\frac{\alpha}{\beta} \right) q \right\}, & 0<q\leq 1+\alpha,\\
\max \left\{ \tau_{X_1^*}(q), \left(1-\frac{\alpha}{\beta} \right) q \right\}, & 1+\alpha < q < \gamma,
\end{cases}
= \left(1-\frac{\alpha}{\beta} \right) q, \quad 0<q<\gamma.
\end{equation*}
\end{proof}

One may follow the proof of Theorem \ref{thm:mainb=0} from Figure \ref{fig4}. Each subfigure shows the scaling function of $X_1^*$ in blue and the scaling function of $X_2^*$ in red. Following Proposition \ref{prop:sfofsum}, the scaling function of the integrated process $X^*$ (thick green) is obtained by taking the maximum of these two functions. The vertical dotted line indicates the range of finite moments of $X_1^*$ and $X^*$. The scaling function of $X^*$ is well-defined only in this range.

\begin{figure}
\centering
\begin{subfigure}[b]{0.45\textwidth}
\newcommand\gammaV{1.2}
\newcommand\alphaV{1.5}
\resizebox{1.1\textwidth}{!}{
\begin{tikzpicture}[domain=0:4]
\begin{axis}[
axis lines=middle,
xlabel=$q$, xlabel style={at=(current axis.right of origin), anchor=west},
ylabel=$\tau_{X^*}(q)$, ylabel style={at=(current axis.above origin), anchor=south},
xtick={0,\gammaV,2,4},
xticklabels={$0$,$\gamma$,$q_*$,$q^*$},
xmin=0,
xmax=5,
ymajorticks=false
]
\addplot[line width=5pt,opacity=0.8,white!70!green,domain=0:\gammaV]{x/\gammaV};
\addplot[thick,white!20!blue,domain=0:\gammaV]{x/\gammaV} node [pos=0.6,left]{$\frac{1}{\gamma}q\ $};
\addplot[thick,white!20!red,domain=0:2]{0.5*x} node [pos=0.7,below]{$\frac{1}{2}q$};
\addplot[thick,white!20!red,domain=4:6]{x-\alphaV} node [pos=0.2,left]{$q-\alpha$};
\addplot[dashed,white!20!red,thick] coordinates {(2,1) (4,2.5)};
\addplot[dashed] coordinates {(0,2.5) (4,2.5)};
\addplot[dashed] coordinates {(4,0) (4,2.5)};
\addplot[dashed] coordinates {(0,1) (2,1)};
\addplot[dashed] coordinates {(2,0) (2,1)};
\addplot[dotted,line width=1.5pt] coordinates {(\gammaV,0) (\gammaV,3.5)};
\end{axis}
\end{tikzpicture}
}
\caption{$\alpha>1$}
\label{fig4a}
\end{subfigure}
\hfill
\begin{subfigure}[b]{0.45\textwidth}
\newcommand\gammaV{0.8}
\resizebox{1.1\textwidth}{!}{
\begin{tikzpicture}[domain=0:3]
\begin{axis}[
axis lines=middle,
xlabel=$q$, xlabel style={at=(current axis.right of origin), anchor=west},
ylabel=$\tau_{X^*}(q)$, ylabel style={at=(current axis.above origin), anchor=south},
xtick={0,\gammaV,1.5},
xticklabels={$0$,$\gamma$,$1+\alpha$},
xmin=0,
xmax=3,
ymajorticks=false
]
\addplot[line width=5pt,opacity=0.8,white!70!green,domain=0:\gammaV]{x/\gammaV};
\addplot[thick,white!20!blue,domain=0:\gammaV]{x/\gammaV} node [pos=0.6,left]{$\frac{1}{\gamma}q\ $};
\addplot[thick,white!20!red,domain=0:1.5]{(1/1.5)*x} node [pos=0.75,below]{$\frac{1}{1+\alpha}q\ $};
\addplot[thick,white!20!red,domain=1.5:4]{x-0.5} node [pos=0.5,left]{$q-\alpha$};
\addplot[dashed] coordinates {(1.5,0) (1.5,1)};
\addplot[dashed] coordinates {(0,1) (1.5,1)};
\addplot[dotted,line width=1.5pt] coordinates {(\gammaV,0) (\gammaV,2.5)};
\end{axis}
\end{tikzpicture}
}
\caption{$\alpha\in(0,1)$, $\gamma<1+\alpha$ and $\beta<1+\alpha$}
\label{fig4b}
\end{subfigure}
\hfill
\begin{subfigure}[b]{0.45\textwidth}
\newcommand\gammaV{1.1}
\resizebox{1.1\textwidth}{!}{
\begin{tikzpicture}[domain=0:3]
\begin{axis}[
axis lines=middle,
xlabel=$q$, xlabel style={at=(current axis.right of origin), anchor=west},
ylabel=$\tau_{X^*}(q)$, ylabel style={at=(current axis.above origin), anchor=south},
xtick={0,\gammaV,1.7},
xticklabels={$0$,$\gamma$,$\beta$},
xmin=0,
xmax=3,
ymajorticks=false
]
\addplot[line width=5pt,opacity=0.8,white!70!green,domain=0:\gammaV]{x/\gammaV};
\addplot[thick,white!20!blue,domain=0:\gammaV]{x/\gammaV} node [pos=0.6,left]{$\frac{1}{\gamma}q\ $};
\addplot[thick,white!20!red,domain=0:1.7]{(1-0.5/1.7)*x} node [pos=0.85,below]{$\left(1-\frac{\alpha}{\beta}\right)q\, $};
\addplot[thick,white!20!red,domain=1.7:3]{x-0.5} node [pos=0.5,left]{$q-\alpha$};
\addplot[dashed] coordinates {(1.7,0) (1.7,1.2)};
\addplot[dashed] coordinates {(0,1.2) (1.7,1.2)};
\addplot[dotted,line width=1.5pt] coordinates {(\gammaV,0) (\gammaV,2.5)};
\end{axis}
\end{tikzpicture}
}
\caption{$\alpha\in(0,1)$, $\gamma<1+\alpha$ and $\beta>1+\alpha$}
\label{fig4c}
\end{subfigure}
\hfill
\begin{subfigure}[b]{0.45\textwidth}
\newcommand\gammaV{1.8}
\newcommand\alphaV{0.2}
\newcommand\alphaVP{1.2}
\resizebox{1.1\textwidth}{!}{
\begin{tikzpicture}[domain=0:3]
\begin{axis}[
axis lines=middle,
xlabel=$q$, xlabel style={at=(current axis.right of origin), anchor=west},
ylabel=$\tau_{X^*}(q)$, ylabel style={at=(current axis.above origin), anchor=south},
xtick={0,\gammaV,1.2},
xticklabels={$0$,$\gamma$,$1+\alpha$},
xmin=0,
xmax=3,
ymajorticks=false
]
\addplot[line width=5pt,opacity=0.8,white!70!green,domain=0:1.2]{(1/1.2)*x};
\addplot[line width=5pt,opacity=0.8,white!70!green,domain=1.2:\gammaV]{x-0.2};
\addplot[thick,white!20!blue,domain=0:(1+\alphaV)]{(1/(1+\alphaV))*x-0.02} node [pos=0.5,above]{$\frac{1}{1+\alpha}q\ $};
\addplot[thick,white!20!blue,dashed,domain=(1+\alphaV):\gammaV]{x-\alphaV-0.02};
\addplot[thick,white!20!red,domain=0:1.2]{(1/1.2)*x};
\addplot[thick,white!20!red,domain=1.2:\gammaV]{x-0.2};
\addplot[thick,white!20!red,domain=\gammaV:3]{x-0.2} node [pos=0.5,left]{$q-\alpha$};
\addplot[dashed] coordinates {(1.2,0) (1.2,1)};
\addplot[dashed] coordinates {(0,1) (1.2,1)};
\addplot[dotted,line width=1.5pt] coordinates {(\gammaV,0) (\gammaV,2.5)};
\end{axis}
\end{tikzpicture}
}
\caption{$\beta<1+\alpha<\gamma$}
\label{fig4d}
\end{subfigure}
\hfill
\begin{subfigure}[b]{0.45\textwidth}
\newcommand\gammaV{1.99}
\newcommand\alphaV{0.4}
\newcommand\alphaVP{1.4}
\resizebox{1.1\textwidth}{!}{
\begin{tikzpicture}[domain=0:3]
\begin{axis}[
axis lines=middle,
xlabel=$q$, xlabel style={at=(current axis.right of origin), anchor=west},
ylabel=$\tau_{X^*}(q)$, ylabel style={at=(current axis.above origin), anchor=south},
xtick={0,\gammaV,1.4,1.85},
xticklabels={$0$,$\quad \gamma$,$1+\alpha$,$\beta$},
xmin=0,
xmax=2.2,
ymin=0,
ymax=1.8,
ymajorticks=false
]
\addplot[line width=5pt,opacity=0.8,white!70!green,domain=0:1.85]{(1-0.4/1.85)*x};
\addplot[line width=5pt,opacity=0.8,white!70!green,domain=1.85:\gammaV]{x-0.4};
\addplot[thick,white!20!blue,domain=0:(1+\alphaV)]{(1/(1+\alphaV))*x} node [pos=0.65,below]{$\frac{1}{1+\alpha}q\ $};
\addplot[thick,white!20!blue,dashed,domain=(1+\alphaV):\gammaV]{x-\alphaV};
\addplot[dashed] coordinates {(1.4,0) (1.4,1)};
\addplot[dashed] coordinates {(0,1) (1.4,1)};
\addplot[thick,white!20!red,domain=0:1.85]{(1-0.4/1.85)*x} node [pos=0.6,left]{$\left(1-\frac{\alpha}{\beta}\right)q\, $};
\addplot[thick,white!20!red,domain=1.85:2.2]{x-0.4} node [pos=0.09,below right]{$q-\alpha$};
\addplot[dashed] coordinates {(1.85,0) (1.85,1.45)};
\addplot[dashed] coordinates {(0,1.45) (1.85,1.45)};
\addplot[dotted,line width=1.5pt] coordinates {(\gammaV,0) (\gammaV,1.8)};
\end{axis}
\end{tikzpicture}
}
\caption{$1+\alpha<\beta\leq\gamma$}
\label{fig4e}
\end{subfigure}
\hfill
\begin{subfigure}[b]{0.45\textwidth}
\newcommand\gammaV{1.65}
\newcommand\alphaV{0.4}
\newcommand\alphaVP{1.4}
\resizebox{1.1\textwidth}{!}{
\begin{tikzpicture}[domain=0:3]
\begin{axis}[
axis lines=middle,
xlabel=$q$, xlabel style={at=(current axis.right of origin), anchor=west},
ylabel=$\tau_{X^*}(q)$, ylabel style={at=(current axis.above origin), anchor=south},
xtick={0,\gammaV,1.4,1.85},
xticklabels={$0$,$\gamma$,$1+\alpha$,$\beta$},
xmin=0,
xmax=2.02,
ymajorticks=false
]
\addplot[line width=5pt,opacity=0.8,white!70!green,domain=0:1.65]{(1-0.4/1.85)*x};
\addplot[thick,white!20!blue,domain=0:(1+\alphaV)]{(1/(1+\alphaV))*x} node [pos=0.65,below]{$\frac{1}{1+\alpha}q\ $};
\addplot[thick,white!20!blue,dashed,domain=(1+\alphaV):\gammaV]{x-\alphaV};
\addplot[dashed] coordinates {(1.4,0) (1.4,1)};
\addplot[dashed] coordinates {(0,1) (1.4,1)};
\addplot[thick,white!20!red,domain=0:1.85]{(1-0.4/1.85)*x} node [pos=0.6,left]{$\left(1-\frac{\alpha}{\beta}\right)q\, $};
\addplot[thick,white!20!red,domain=1.85:2.05]{x-0.4} node [pos=0.06,above left]{$q-\alpha$};
\addplot[dashed] coordinates {(1.85,0) (1.85,1.45)};
\addplot[dashed] coordinates {(0,1.45) (1.85,1.45)};
\addplot[dotted,line width=1.5pt] coordinates {(\gammaV,0) (\gammaV,1.6)};
\end{axis}
\end{tikzpicture}
}
\caption{$1+\alpha<\gamma<\beta$}
\label{fig4f}
\end{subfigure}
\caption{The scaling functions of $X^*$ when $b=0$ (no Gaussian component). Each plot shows the scaling functions $\tau_{X_1^*}$ (blue), $\tau_{X_2^*}$ (red) and $\tau_{X^*}$ (thick green). Dashed parts of the plots denote the upper bounds. The vertical thick dotted line denotes the position of $\gamma$, beyond which the moments of $X^*_1$ and $X^*$ are infinite.}\label{fig4}
\end{figure}
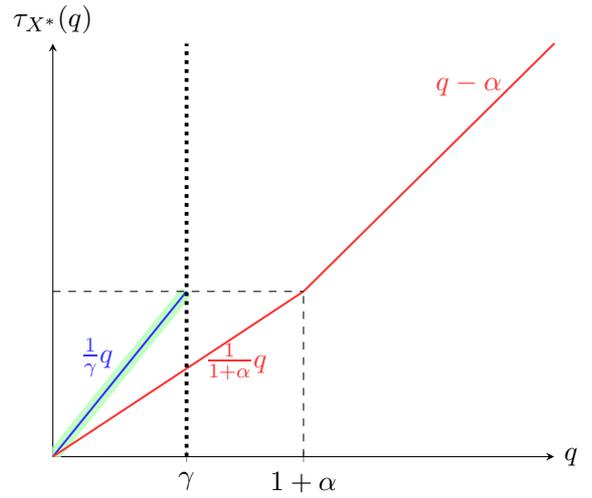
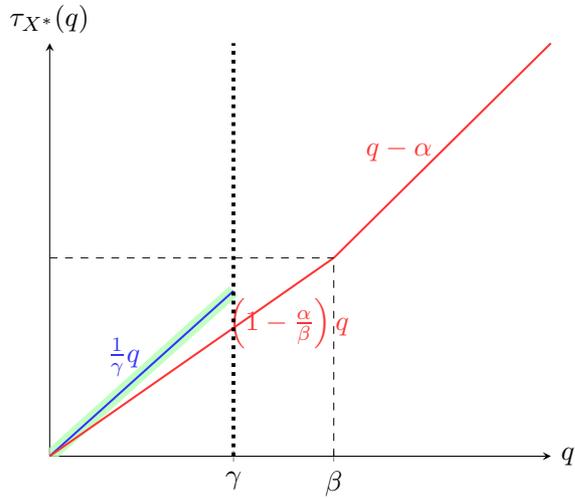
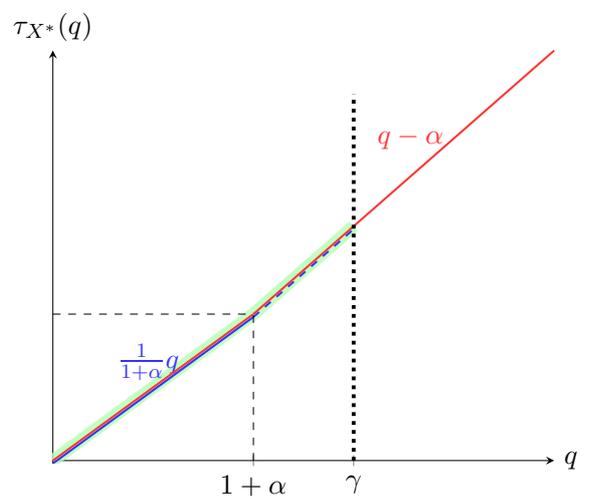
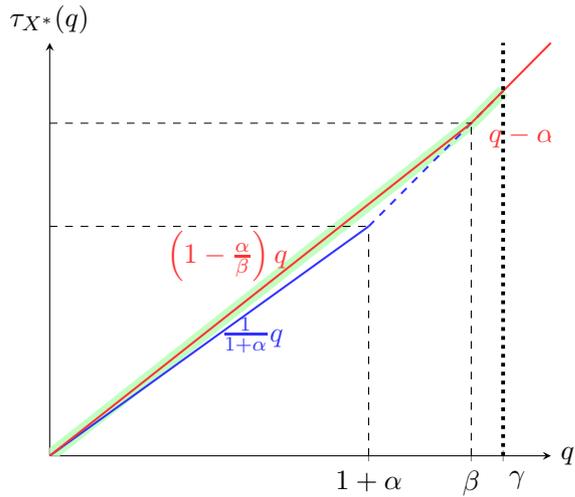
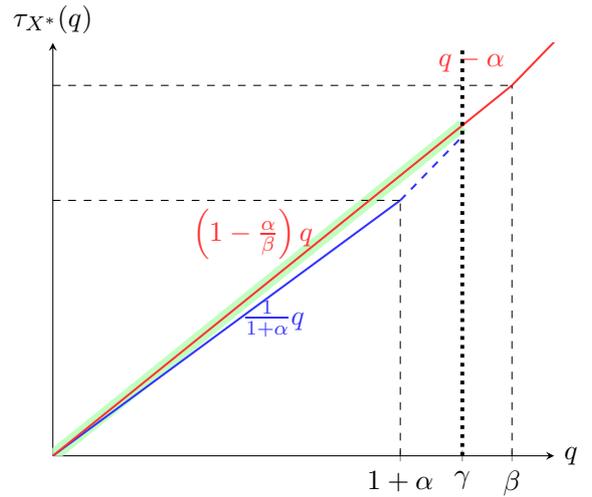

\begin{proof}[Proof of Theorem \ref{thm:mainb!=0}]
We will use the results of Theorem \ref{thm:mainb=0} and Lemma \ref{lemma:X3} and combine them using Proposition \ref{prop:sfofsum} so that
\begin{equation*}
\tau_{X^*}(q) = \max \left\{ \tau_{X_1^*+X_2^*}(q), \tau_{X_3^*}(q) \right\}.
\end{equation*}

\textit{(a)} If $\alpha>1$, then for $q<\gamma$
\begin{equation*}
\tau_{X^*}(q)=\max \left\{ \frac{1}{\gamma}q,  \frac{1}{2}q \right\} = \frac{1}{\gamma}q.
\end{equation*}
If $\alpha \in (0,1)$ and $\gamma <\frac{2}{2-\alpha}$, then also $\gamma<1+\frac{\alpha}{2-\alpha}<1+\alpha$ and hence
\begin{equation*}
\tau_{X^*}(q)=\max \left\{ \frac{1}{\gamma}q,  \left(1-\frac{\alpha}{2} \right) q \right\} = \frac{1}{\gamma}q,
\end{equation*}
since $1/\gamma >1-\alpha/2 \Leftrightarrow \gamma < 2/(2-\alpha)$.

\textit{(b)} Suppose now that $\alpha \in (0,1)$ and $\gamma >\frac{2}{2-\alpha}$.
\begin{itemize}
	\item If $\frac{2}{2-\alpha} < \gamma < 1+\alpha$, then
	\begin{equation*}
	\tau_{X^*}(q)=\max \left\{ \frac{1}{\gamma}q,  \left(1-\frac{\alpha}{2} \right) q \right\} = \left(1-\frac{\alpha}{2} \right) q,
	\end{equation*}
	since $1/\gamma <1-\alpha/2$.
	\item If $\gamma>1+\alpha$ and $\beta<1+\alpha$, then we have
	\begin{equation*}
	\tau_{X^*}(q) = \begin{cases}
	\max \left\{ \frac{1}{1+\alpha} q, \left(1-\frac{\alpha}{2} \right) q \right\}, & 0<q\leq 1+\alpha,\\
	\max \left\{ q-\alpha, \left(1-\frac{\alpha}{2} \right) q \right\}, & 1+\alpha < q < \gamma.
	\end{cases}
	\end{equation*}
	Now $\alpha<1$ implies $\frac{1}{1+\alpha}= 1- \frac{\alpha}{1+\alpha}< 1-\frac{\alpha}{2}$ and for $q<2$ we have $q- \frac{\alpha}{2} q >q-\alpha$. Hence,
	\begin{equation*}
	\tau_{X^*}(q)=\left(1-\frac{\alpha}{2} \right) q, \quad 0<q<\gamma.
	\end{equation*}
	\item If $\gamma>1+\alpha$, $1+\alpha<\beta$ and $\beta \leq \gamma$, then
	\begin{equation*}
	\tau_{X^*}(q) = \begin{cases}
	\max \left\{ \left(1-\frac{\alpha}{\beta} \right) q, \left(1-\frac{\alpha}{2} \right) q \right\}, & 0<q\leq \beta,\\
	\max \left\{ q-\alpha, \left(1-\frac{\alpha}{2} \right) q \right\}, & \beta < q < \gamma.
	\end{cases}
	=\left(1-\frac{\alpha}{2} \right) q, \ 0<q<\gamma,
	\end{equation*}	
	since $1-\frac{\alpha}{\beta}<1-\frac{\alpha}{2}$ and by the same argument as in the previous case.
	\item The same argument applies to case $\gamma>1+\alpha$, $1+\alpha<\beta$ and $\beta > \gamma$.
\end{itemize}
\end{proof}

Figures \ref{fig5} and \ref{fig6} illustrate the proof of Theorem \ref{thm:mainb!=0}. The scaling functions $\tau_{X_1^*}$, $\tau_{X_2^*}$ and $\tau_{X_3^*}$ of each component are shown on each plot in red, blue and purple, respectively, while their maximum is denoted by the thick green line. Figure \ref{fig5} is related to the case (a) of Theorem \ref{thm:mainb!=0} and Figure \ref{fig6} to the case (b) of Theorem \ref{thm:mainb!=0}. The figures are split based on different forms of the scaling functions of the three components $X_1^*$, $X_2^*$ and $X_3^*$.

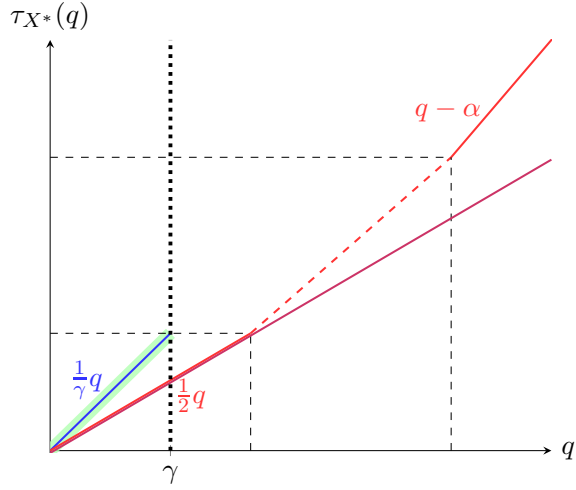
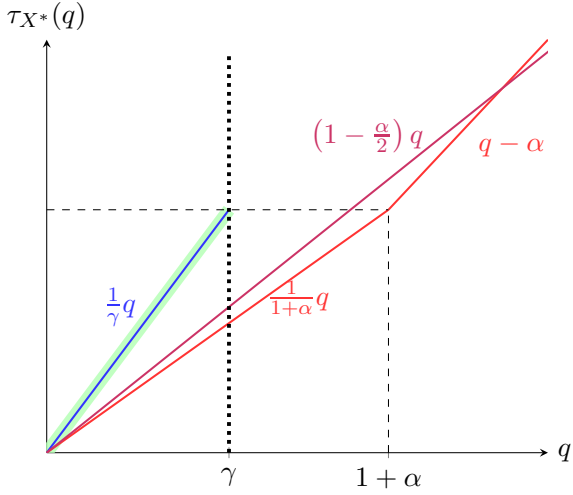
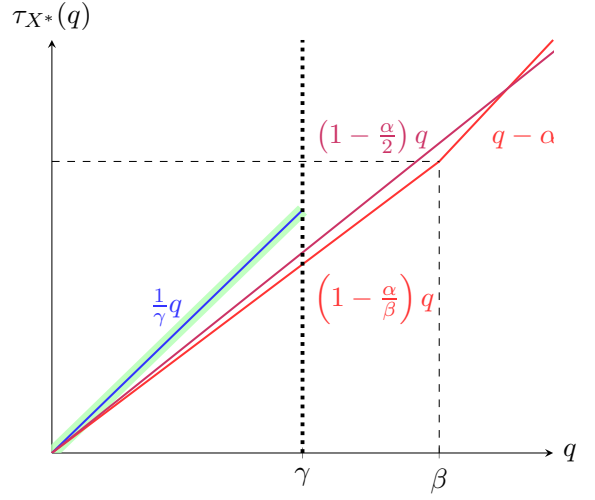
\begin{figure}
\centering
\begin{subfigure}[b]{0.45\textwidth}
\newcommand\gammaV{1.2}
\newcommand\alphaV{1.5}
\resizebox{1.1\textwidth}{!}{
\begin{tikzpicture}[domain=0:4]
\begin{axis}[
axis lines=middle,
xlabel=$q$, xlabel style={at=(current axis.right of origin), anchor=west},
ylabel=$\tau_{X^*}(q)$, ylabel style={at=(current axis.above origin), anchor=south},
xtick={0,\gammaV},
xticklabels={$0$,$\gamma$},
xmin=0,
xmax=5,
ymajorticks=false
]
\addplot[line width=5pt,opacity=0.8,white!70!green,domain=0:\gammaV]{x/\gammaV};
\addplot[thick,white!20!blue,domain=0:\gammaV]{x/\gammaV} node [pos=0.6,left]{$\frac{1}{\gamma}q\ $};
\addplot[thick,white!20!red,domain=0:2]{0.5*x} node [pos=0.7,below]{$\frac{1}{2}q$};
\addplot[thick,white!20!red,domain=4:6]{x-\alphaV} node [pos=0.2,left]{$q-\alpha$};
\addplot[dashed,white!20!red,thick] coordinates {(2,1) (4,2.5)};
\addplot[dashed] coordinates {(0,2.5) (4,2.5)};
\addplot[dashed] coordinates {(4,0) (4,2.5)};
\addplot[dashed] coordinates {(0,1) (2,1)};
\addplot[dashed] coordinates {(2,0) (2,1)};
\addplot[thick,white!20!purple,domain=0:5]{x/2-0.02};
\addplot[dotted,line width=1.5pt] coordinates {(\gammaV,0) (\gammaV,3.5)};
\end{axis}
\end{tikzpicture}
}
\caption{$\alpha>1$}
\label{fig5a}
\end{subfigure}
\hfill\\
\begin{subfigure}[b]{0.45\textwidth}
\newcommand\gammaV{0.8}
\newcommand\alphaV{0.5}
\resizebox{1.1\textwidth}{!}{
\begin{tikzpicture}[domain=0:3]
\begin{axis}[
axis lines=middle,
xlabel=$q$, xlabel style={at=(current axis.right of origin), anchor=west},
ylabel=$\tau_{X^*}(q)$, ylabel style={at=(current axis.above origin), anchor=south},
xtick={0,\gammaV,1.5},
xticklabels={$0$,$\gamma$,$1+\alpha$},
xmin=0,
xmax=2.2,
ymajorticks=false
]
\addplot[line width=5pt,opacity=0.8,white!70!green,domain=0:\gammaV]{x/\gammaV};
\addplot[thick,white!20!blue,domain=0:\gammaV]{x/\gammaV} node [pos=0.6,left]{$\frac{1}{\gamma}q\ $};
\addplot[thick,white!20!red,domain=0:1.5]{(1/1.5)*x} node [pos=0.75,below]{$\frac{1}{1+\alpha}q\ $};
\addplot[thick,white!20!red,domain=1.5:2.2]{x-0.5} node [pos=0.5,below right]{$q-\alpha$};
\addplot[dashed] coordinates {(1.5,0) (1.5,1)};
\addplot[dashed] coordinates {(0,1) (1.5,1)};
\addplot[thick,white!20!purple,domain=0:3]{(1-\alphaV/2)*x} node [pos=0.58,left]{$\left(1-\frac{\alpha}{2}\right)q\ $};
\addplot[dotted,line width=1.5pt] coordinates {(\gammaV,0) (\gammaV,1.65)};
\end{axis}
\end{tikzpicture}
}
\caption{$\alpha\in(0,1)$, $\gamma<\frac{2}{2-\alpha}<1+\alpha$ and $\beta<1+\alpha$}
\label{fig5b}
\end{subfigure}
\hfill
\begin{subfigure}[b]{0.45\textwidth}
\newcommand\gammaV{1.1}
\resizebox{1.1\textwidth}{!}{
\begin{tikzpicture}[domain=0:3]
\begin{axis}[
axis lines=middle,
xlabel=$q$, xlabel style={at=(current axis.right of origin), anchor=west},
ylabel=$\tau_{X^*}(q)$, ylabel style={at=(current axis.above origin), anchor=south},
xtick={0,\gammaV,1.7},
xticklabels={$0$,$\gamma$,$\beta$},
xmin=0,
xmax=2.2,
ymajorticks=false
]
\addplot[line width=5pt,opacity=0.8,white!70!green,domain=0:\gammaV]{x/\gammaV};
\addplot[thick,white!20!blue,domain=0:\gammaV]{x/\gammaV} node [pos=0.6,left]{$\frac{1}{\gamma}q\ $};
\addplot[thick,white!20!red,domain=0:1.7]{(1-0.5/1.7)*x} node [pos=0.65,below right]{$\left(1-\frac{\alpha}{\beta}\right)q\, $};
\addplot[thick,white!20!red,domain=1.7:2.2]{x-0.5} node [pos=0.37,below right]{$q-\alpha$};
\addplot[dashed] coordinates {(1.7,0) (1.7,1.2)};
\addplot[dashed] coordinates {(0,1.2) (1.7,1.2)};
\addplot[thick,white!20!purple,domain=0:3]{(1-0.5/2)*x} node [pos=0.58,left]{$\left(1-\frac{\alpha}{2}\right)q\ $};
\addplot[dotted,line width=1.5pt] coordinates {(\gammaV,0) (\gammaV,1.7)};
\end{axis}
\end{tikzpicture}
}
\caption{$\alpha\in(0,1)$, $\gamma<\frac{2}{2-\alpha}<1+\alpha$ and $\beta>1+\alpha$}
\label{fig5c}
\end{subfigure}
\caption{The scaling functions of $X^*$ when $b\neq 0$: case (a) of Theorem \ref{thm:mainb!=0}. Each plot shows the scaling functions $\tau_{X_1^*}$ (blue), $\tau_{X_2^*}$ (red), $\tau_{X_3^*}$ (purple) and $\tau_{X^*}$ (thick green). Dashed part of the plot denotes the upper bound. The vertical thick dotted line denotes the position of $\gamma$, beyond which the moments of $X^*_1$ and $X^*$ are infinite. }\label{fig5}
\end{figure}

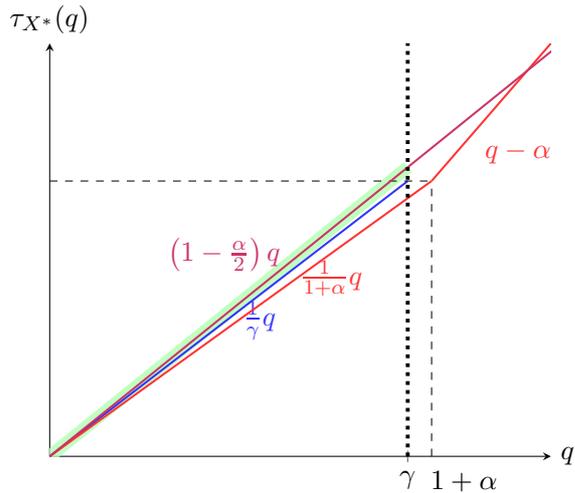
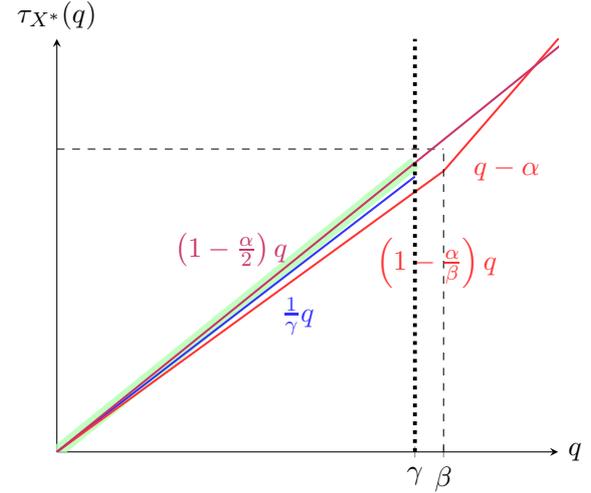
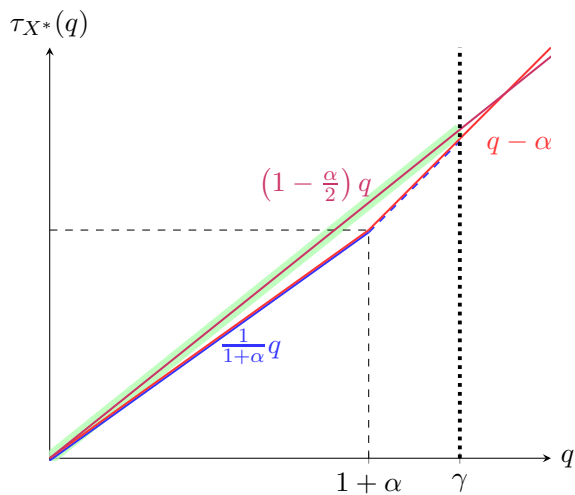
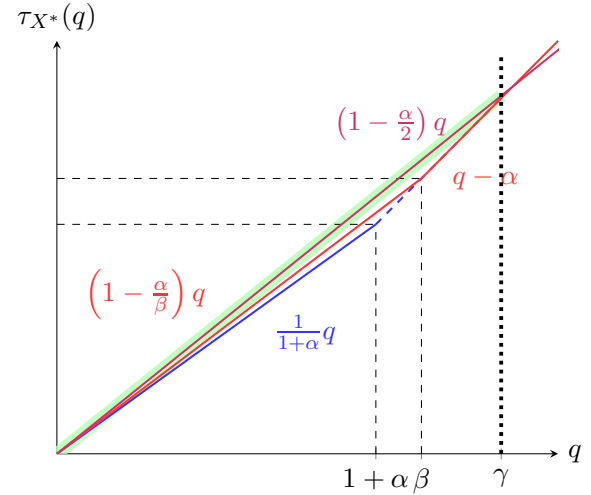
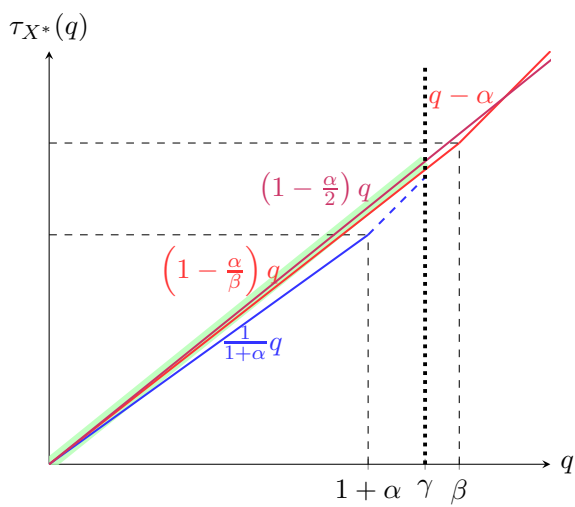
\begin{figure}
\centering
\begin{subfigure}[b]{0.45\textwidth}
\newcommand\gammaV{1.5}
\newcommand\alphaV{0.6}
\resizebox{1.1\textwidth}{!}{
\begin{tikzpicture}[domain=0:3]
\begin{axis}[
axis lines=middle,
xlabel=$q$, xlabel style={at=(current axis.right of origin), anchor=west},
ylabel=$\tau_{X^*}(q)$, ylabel style={at=(current axis.above origin), anchor=south},
xtick={0,\gammaV,1.5},
xticklabels={$0$,$\gamma$,$\qquad \qquad 1+\alpha$},
xmin=0,
xmax=2.1,
ymajorticks=false
]
\addplot[line width=5pt,opacity=0.8,white!70!green,domain=0:\gammaV]{(1-\alphaV/2)*x};
\addplot[thick,white!20!blue,domain=0:\gammaV]{x/\gammaV} node [pos=0.6,below]{$\frac{1}{\gamma}q\ $};
\addplot[thick,white!20!red,domain=0:1.6]{(1/1.6)*x} node [pos=0.75,below]{$\frac{1}{1+\alpha}q\ $};
\addplot[thick,white!20!red,domain=1.6:2.2]{x-0.6} node [pos=0.3,below right]{$q-\alpha$};
\addplot[dashed] coordinates {(1.6,0) (1.6,1)};
\addplot[dashed] coordinates {(0,1) (1.6,1)};
\addplot[thick,white!20!purple,domain=0:3]{(1-\alphaV/2)*x} node [pos=0.35,left]{$\left(1-\frac{\alpha}{2}\right)q\ $};
\addplot[dotted,line width=1.5pt] coordinates {(\gammaV,0) (\gammaV,1.5)};
\end{axis}
\end{tikzpicture}
}
\caption{$\alpha\in(0,1)$, $\frac{2}{2-\alpha}<\gamma<1+\alpha$ and $\beta<1+\alpha$}
\label{fig6a}
\end{subfigure}
\hfill
\begin{subfigure}[b]{0.45\textwidth}
\newcommand\gammaV{1.5}
\newcommand\alphaV{0.6}
\resizebox{1.1\textwidth}{!}{
\begin{tikzpicture}[domain=0:3]
\begin{axis}[
axis lines=middle,
xlabel=$q$, xlabel style={at=(current axis.right of origin), anchor=west},
ylabel=$\tau_{X^*}(q)$, ylabel style={at=(current axis.above origin), anchor=south},
xtick={0,\gammaV,1.62},
xticklabels={$0$,$\gamma$,$\beta$},
xmin=0,
xmax=2.1,
ymajorticks=false
]
\addplot[line width=5pt,opacity=0.8,white!70!green,domain=0:\gammaV]{(1-0.6/2)*x};
\addplot[thick,white!20!blue,domain=0:\gammaV]{x/\gammaV} node [pos=0.6,below right]{$\frac{1}{\gamma}q\ $};
\addplot[thick,white!20!red,domain=0:1.62]{(1-0.6/1.62)*x} node [pos=0.8,below right]{$\left(1-\frac{\alpha}{\beta}\right)q\, $};
\addplot[thick,white!20!red,domain=1.62:3]{x-0.6} node [pos=0.06,below right]{$q-\alpha$};
\addplot[dashed] coordinates {(1.62,0) (1.62,1.1)};
\addplot[dashed] coordinates {(0,1.1) (1.62,1.1)};
\addplot[thick,white!20!purple,domain=0:3]{(1-0.6/2)*x} node [pos=0.35,left]{$\left(1-\frac{\alpha}{2}\right)q\ $};
\addplot[dotted,line width=1.5pt] coordinates {(\gammaV,0) (\gammaV,1.5)};
\end{axis}
\end{tikzpicture}
}
\caption{$\alpha\in(0,1)$, $\frac{2}{2-\alpha}<\gamma<1+\alpha$ and $\beta>1+\alpha$}
\label{fig6b}
\end{subfigure}
\hfill
\begin{subfigure}[b]{0.45\textwidth}
\newcommand\gammaV{1.8}
\newcommand\alphaV{0.4}
\newcommand\alphaVP{1.4}
\resizebox{1.1\textwidth}{!}{
\begin{tikzpicture}[domain=0:3]
\begin{axis}[
axis lines=middle,
xlabel=$q$, xlabel style={at=(current axis.right of origin), anchor=west},
ylabel=$\tau_{X^*}(q)$, ylabel style={at=(current axis.above origin), anchor=south},
xtick={0,\gammaV,1.4},
xticklabels={$0$,$\gamma$,$1+\alpha$},
xmin=0,
xmax=2.2,
ymajorticks=false
]
\addplot[line width=5pt,opacity=0.8,white!70!green,domain=0:1.8]{(1-\alphaV/2)*x};
\addplot[thick,white!20!blue,domain=0:(1+\alphaV)]{(1/(1+\alphaV))*x-0.01} node [pos=0.5,right]{$\frac{1}{1+\alpha}q\ $};
\addplot[thick,white!20!blue,dashed,domain=(1+\alphaV):\gammaV]{x-\alphaV-0.01};
\addplot[thick,white!20!red,domain=0:1.4]{(1/1.4)*x};
\addplot[thick,white!20!red,domain=1.4:\gammaV]{x-0.4};
\addplot[thick,white!20!red,domain=\gammaV:2.2]{x-0.4} node [pos=0.18,below right]{$q-\alpha$};
\addplot[dashed] coordinates {(1.4,0) (1.4,1)};
\addplot[dashed] coordinates {(0,1) (1.4,1)};
\addplot[thick,white!20!purple,domain=0:3]{(1-\alphaV/2)*x} node [pos=0.5,left]{$\left(1-\frac{\alpha}{2}\right)q\ $};
\addplot[dotted,line width=1.5pt] coordinates {(\gammaV,0) (\gammaV,1.8)};
\end{axis}
\end{tikzpicture}
}
\caption{$\beta<1+\alpha<\gamma$ (implies $\alpha \in (0,1)$ and $\gamma > \frac{2}{2-\alpha}$)}
\label{fig6c}
\end{subfigure}
\hfill
\begin{subfigure}[b]{0.45\textwidth}
\newcommand\gammaV{1.95}
\newcommand\alphaV{0.4}
\newcommand\alphaVP{1.4}
\resizebox{1.1\textwidth}{!}{
\begin{tikzpicture}[domain=0:3]
\begin{axis}[
axis lines=middle,
xlabel=$q$, xlabel style={at=(current axis.right of origin), anchor=west},
ylabel=$\tau_{X^*}(q)$, ylabel style={at=(current axis.above origin), anchor=south},
xtick={0,\gammaV,1.4,1.6},
xticklabels={$0$,$\gamma$,$1+\alpha$,$\beta$},
xmin=0,
xmax=2.2,
ymajorticks=false
]
\addplot[line width=5pt,opacity=0.8,white!70!green,domain=0:1.95]{(1-\alphaV/2)*x};
\addplot[thick,white!20!blue,domain=0:(1+\alphaV)]{(1/(1+\alphaV))*x} node [pos=0.65,below right]{$\frac{1}{1+\alpha}q\ $};
\addplot[thick,white!20!blue,dashed,domain=(1+\alphaV):\gammaV]{x-\alphaV};
\addplot[dashed] coordinates {(1.4,0) (1.4,1)};
\addplot[dashed] coordinates {(0,1) (1.4,1)};
\addplot[thick,white!20!red,domain=0:1.6]{(1-0.4/1.6)*x} node [pos=0.45,above left]{$\left(1-\frac{\alpha}{\beta}\right)q\, $};
\addplot[thick,white!20!red,domain=1.6:2.2]{x-0.4} node [pos=0.15,below right]{$q-\alpha$};
\addplot[dashed] coordinates {(1.6,0) (1.6,1.2)};
\addplot[dashed] coordinates {(0,1.2) (1.6,1.2)};
\addplot[thick,white!20!purple,domain=0:3]{(1-\alphaV/2)*x} node [pos=0.6,left]{$\left(1-\frac{\alpha}{2}\right)q\ $};
\addplot[dotted,line width=1.5pt] coordinates {(\gammaV,0) (\gammaV,1.75)};
\end{axis}
\end{tikzpicture}
}
\caption{$1+\alpha<\beta\leq\gamma$ (implies $\alpha \in (0,1)$ and $\gamma > \frac{2}{2-\alpha}$)}
\label{fig6d}
\end{subfigure}
\hfill
\begin{subfigure}[b]{0.45\textwidth}
\newcommand\gammaV{1.65}
\newcommand\alphaV{0.4}
\newcommand\alphaVP{1.4}
\resizebox{1.1\textwidth}{!}{
\begin{tikzpicture}[domain=0:3]
\begin{axis}[
axis lines=middle,
xlabel=$q$, xlabel style={at=(current axis.right of origin), anchor=west},
ylabel=$\tau_{X^*}(q)$, ylabel style={at=(current axis.above origin), anchor=south},
xtick={0,\gammaV,1.4,1.8},
xticklabels={$0$,$\gamma$,$1+\alpha$,$\beta$},
xmin=0,
xmax=2.2,
ymajorticks=false
]
\addplot[line width=5pt,opacity=0.8,white!70!green,domain=0:1.65]{(1-\alphaV/2)*x};
\addplot[thick,white!20!blue,domain=0:(1+\alphaV)]{(1/(1+\alphaV))*x} node [pos=0.65,below]{$\frac{1}{1+\alpha}q\ $};
\addplot[thick,white!20!blue,dashed,domain=(1+\alphaV):\gammaV]{x-\alphaV};
\addplot[dashed] coordinates {(1.4,0) (1.4,1)};
\addplot[dashed] coordinates {(0,1) (1.4,1)};
\addplot[thick,white!20!red,domain=0:1.8]{(1-0.4/1.8)*x} node [pos=0.6,left]{$\left(1-\frac{\alpha}{\beta}\right)q\, $};
\addplot[thick,white!20!red,domain=1.8:2.2]{x-0.4} node [pos=0.5,left]{$q-\alpha$};
\addplot[dashed] coordinates {(1.8,0) (1.8,1.4)};
\addplot[dashed] coordinates {(0,1.4) (1.8,1.4)};
\addplot[thick,white!20!purple,domain=0:3]{(1-\alphaV/2)*x} node [pos=0.5,left]{$\left(1-\frac{\alpha}{2}\right)q\ $};
\addplot[dotted,line width=1.5pt] coordinates {(\gammaV,0) (\gammaV,1.75)};
\end{axis}
\end{tikzpicture}
}
\caption{$1+\alpha<\gamma<\beta$ (implies $\alpha \in (0,1)$ and $\gamma > \frac{2}{2-\alpha}$)}
\label{fig6e}
\end{subfigure}
\caption{The scaling functions of $X^*$ when $b\neq 0$: case (b) of Theorem \ref{thm:mainb!=0}. Each plot shows the scaling functions $\tau_{X_1^*}$ (blue), $\tau_{X_2^*}$ (red), $\tau_{X_3^*}$ (purple) and $\tau_{X^*}$ (thick green). Dashed part of the plot denotes the upper bound. The vertical thick dotted line denotes the position of $\gamma$, beyond which the moments of $X^*_1$ and $X^*$ are infinite.}\label{fig6}
\end{figure}

\section*{Acknowledgements}
Nikolai N.~Leonenko was supported in particular by Cardiff Incoming Visiting Fellowship Scheme, International Collaboration Seedcorn Fund, Australian Research Council's Discovery Projects funding scheme (project DP160101366) and the project MTM2015-71839-P of MINECO, Spain (co-funded with FEDER funds). Murad S.~Taqqu was supported in part by the Simons foundation grant 569118 at Boston University. Danijel Grahovac was partially supported by the University of Osijek Grant ZUP2018-31.

\bigskip
\bigskip

\bibliographystyle{agsm}
\bibliography{References}

\end{document}